\documentclass[10pt]{amsart}
\usepackage{pre}
\title[Asymptotics of Convex Core Volumes]{On the Asymptotics of Convex Core Volumes of Once-Punctured Torus Groups}
\author{Hidetoshi Masai}
\address{Humanities and Sciences/Museum Careers, Musashino Art University
Bldg. 12, 1-736 Ogawa-cho, Kodaira-shi, Tokyo 187-8505}
\email{hmasai 'at' musabi.ac.jp}
\thanks{This work was partially supported by JSPS KAKENHI Grant Number 23K03085.}

\author{Ryo Matsuda}
\address{Ritsumeikan University, Biwako-Kusatsu Campus, 1-1-1 Noji-higashi, Kusatsu, Shiga 525-8577, Japan}
\email{r-mat 'at' fc.ritsumei.ac.jp}

\begin{document}
\begin{abstract}
	We study the asymptotic behavior of the convex core volume 
    for a sequence of quasi-Fuchsian manifolds $Q(\psi^{-n}X, \psi^n X)$ 
	associated with a pseudo-Anosov mapping class $\psi$ on a once-punctured torus. 
	We prove that the volume of the convex core differs from 
	$2n$ times the volume of the mapping torus of $\psi$ 
	by at most a uniformly bounded constant.
\end{abstract}

\maketitle
\tableofcontents
\section{Introduction}
Quasi-Fuchsian manifolds are fundamental objects that bridge two-dimensional and three-dimensional hyperbolic geometry.
According to Bers' simultaneous uniformization theorem, 
for any two points $X, Y$ in the Teichm\"uller space $\T(S)$ of an orientable surface $S$ of finite type,
there exists a 3-dimensional manifold $Q(X, Y)$, called a quasi-Fuchsian manifold.  

The deformation theory of $Q(X, Y)$ with the two-dimensional hyperbolic structures $X, Y$ as parameters has been a central topic in the study of Kleinian groups and hyperbolic 3-manifolds.

The mapping class group $\mcg(S)$ acts on $\T(S)$. 
Hence, for a pseudo-Anosov map $\psi: S \to S$, 
we can consider the sequences of points $\psi^{-n}X$ and $\psi^n X$ in the Teichm\"uller space. 
By the results of Thurston and others (see, e.g. McMullen \cite{McM96}), 
the corresponding sequence of quasi-Fuchsian manifolds $Q_n := Q(\psi^{-n}X, \psi^n X)$ 
is known to converge both geometrically and algebraically to the infinite cyclic covering of the mapping torus 
$M_\psi := S \times [0, 1] / (x, 1) \sim (\psi(x), 0)$. 

In this paper, we focus on the convex core volume of such quasi-Fuchsian manifolds $(Q_n)$, 
and investigate its asymptotic behavior. 
More specifically, we treat the case where $S$ is a once-punctured torus or a four-punctured sphere. 
Let $\core(Q_n)$ denote the convex core of $Q_n$. 
We prove that the volume of $\core(Q_n)$ 
is asymptotically comparable with $2n$ times the hyperbolic volume 
$\Vol(M_\psi)$ of the mapping torus $M_\psi$. 
More precisely, our main result is stated as follows.

\begin{thm}\label{thm:main}
Let $S$ be a once-punctured torus or a four-punctured sphere and let $\psi \in \mcg$ be a pseudo-Anosov mapping class. 
For any $X \in \T$, let $Q_n = Q(\psi^{-n}X, \psi^n X)$ be the sequence of quasi-Fuchsian manifolds. 
Then, there exists a constant $C > 0$ independent of $n$ such that
\begin{equation}
   | \Vol(\core(Q_n)) - 2n \Vol(M_\psi) | = O(1) \quad (\text{as\ } n \to \infty ). \label{eq:main_result}
\end{equation}
\end{thm}

When $S$ is an orientable {\em closed} surface, 
the comparison result \eqref{eq:main_result} has already been proved 
by Kojima--McShane \cite{KM18} and Brock--Bromberg \cite{BB16}. 
Although Theorem~\ref{thm:main} may be viewed as a punctured analogue of the results of Kojima--McShane and Brock--Bromberg, our proof should not be regarded as an extension of the methods of \cite{KM18,BB16}.

The geometric inflexibility arguments used in  \cite{KM18, BB16}
rely on the compactness of the convex core.
When $S$ has a puncture and the convex core loses compactness, 
one standard approach is to truncate the portion corresponding to the cusp using a horoball, 
dividing the manifold into a thick part and a thin part to evaluate the geometry. 
Unfortunately, it is difficult to uniformly control the variation of the geometric structure inside the cusp neighborhood as the number of iterations $n$ increases, 
and no geometric inflexibility theorem sufficient for our purposes is presently available in the cusped setting.

Therefore, to overcome this difficulty, we focus on the once-punctured torus (the four-punctured sphere case follows by commensurability).
The primary reason for selecting the once-punctured torus is that it has a balance between generality and tractability:
while preserving the essential features of general non-compact surfaces, 
its combinatorial and topological structures are relatively simple,
permitting concrete analysis accompanied by some explicit combinatorial descriptions.
Historically, Minsky's foundational work on the once-punctured torus \cite{Min99} 
served as the starting point for the resolution of the ending lamination conjecture, which was completed by Brock, Canary, and Minsky \cite{BCM12}. 

In this paper, instead of extending geometric inflexibility to the cusped setting, our approach is based on Gu\'eritaud's canonical triangulations of quasi-Fuchsian once-punctured torus groups \cite{Gue06,Gue09}.
With the aid of the work of Keen--Series \cite{KeeSer04} on the pleating coordinates of quasi-Fuchsian once-punctured torus groups, we first observe that large portions of Gu\'eritaud's triangulations of the convex core of $Q_n$ are combinatorially identical to those of the mapping torus $M_\psi$ (Theorem \ref{thm:combinatorial_inflexibility}).
This {\em combinatorial inflexibility} plays the role of geometric inflexibility in the cusped setting.

A novel aspect of Gu\'eritaud's method is that the geometric realization of the convex core 
is uniquely characterized by a ``volume maximization principle'' 
within a parameter space of dihedral angles determined by the combinatorics of the Farey tessellation. 
We obtain upper and lower bounds for the convex core volume by evaluating 
how the geometric realization corresponding to approximately $2n$ copies 
of the fundamental domain of $M_\psi$ can be constructively embedded, and vice versa.
Our approach is independent of the compactness arguments used in \cite{BB16, KM18}, and
gives a new understanding of the asymptotic behavior of the volume of the convex core in the punctured torus case.

\section*{Acknowledgements}
The authors would like to express their sincere gratitude to  
Hirotaka Akiyoshi, 
Sadayoshi Kojima, 
Greg McShane, 
Makoto Sakuma, and
Yasushi Yamashita
for their invaluable comments, insightful discussions, and constant support throughout the preparation of this manuscript.

\section{Preliminaries} 

\subsection{Notations}

\renewcommand{\arraystretch}{1.3}

\begin{longtable}{r @{\quad : \quad} p{10cm}}
    $\To, \Tf$
    & Teichm\"uller spaces of $S_{1,1}$ and $S_{0,4}$ respectively. \\

    $S$
    & Either $S_{1,1}$ or $S_{0,4}$, unless otherwise stated. \\

    $\psi$
    & A pseudo-Anosov map $\psi: S \to S$ with stable and unstable laminations
      $\Lambda_{\pm}$. \\

    $[\mu]$
    & The projective class of a measured lamination $\mu$. \\

    $M_\psi$
    & The mapping torus defined as
      $S \times [0,1]/(x,1)\sim(\psi(x),0)$. \\

    \noalign{\smallskip}
    \hline
    \noalign{\smallskip}

    $\QF(S)$
    & The space of quasi-Fuchsian representations. \\

    $Q(X,Y)$
    & The quasi-Fuchsian manifold determined by $X,Y\in\T$. \\

    $Q_{n,m}$
    & The manifold $Q(\psi^n(X),\psi^m(X))$, allowing
      $n,m=\pm\infty$. \\

    $\operatorname{core}(Q)$
    & The convex core of the quasi-Fuchsian manifold $Q$. \\

    $\pl_{\pm}(Q)$
    & The pleating laminations
      (see Figure~\ref{fig:pleatinglami}). \\

    $Q_n$
    & The manifold
      $Q(\psi^{-n}(X),\psi^n(X))$. \\

    $\pl_{\pm,n}$
    & The pleating laminations of
      $Q_n=Q(\psi^{-n}(X),\psi^n(X))$. \\

    \noalign{\smallskip}
    \hline
    \noalign{\smallskip}

    $\ks$
    & The Keen--Series pleating coordinates, given by the map: \\

    \multicolumn{1}{r}{}
    & $\QF\setminus\Delta
       \longrightarrow
       (\pml\times\mathbb{R}_{>0})^2$. \\

    $F$
    & The map $F$ from the real line $\mathbb{R}$ to the Farey tree. \\

    $[\mu,\nu]$
    & The geodesic in the Farey tree connecting
      $F([\mu])$ and $F([\nu])$.
\end{longtable}

\subsection{Teichm\"uller space and the mapping class group}\label{sec.TMB}

Let $S = S_{g,n}$ be an oriented topological surface of finite type with genus $g$ and $n$ punctures. 
Throughout this paper, we are mainly concerned with the surfaces $S_{1,1}$ and $S_{0,4}$. 
A marked Riemann surface is a pair $(X, f)$, 
where $X$ is a Riemann surface and $f \colon S \to X$ is an orientation-preserving homeomorphism. 
The Teichm\"uller space $\T(S)$ is defined as the set of equivalence classes $[X, f]$, 
where two pairs $(X_1, f_1)$ and $(X_2, f_2)$ are equivalent 
if there exists a biholomorphism $h \colon X_1 \to X_2$ 
such that $h\circ f_1$ is isotopic to $f_2$ relative to the punctures. 
By abuse of notation, 
we often write the equivalence class $[X, f]$ simply as $X$ 
when the marking is understood from the context.

The mapping class group $\mcg(S) = \mathrm{Homeo}^+(S) / \mathrm{Homeo}_0(S)$ 
acts naturally on $\T(S)$ by changing the marking. 
For an element $[\psi] \in \mcg(S)$ represented 
by a homeomorphism $\psi \colon S \to S$, 
its action on $\T(S)$ is given by $[\psi] \cdot [X, f] = [X, f \circ \psi^{-1}]$.

\subsection{Measured laminations and the Thurston compactification}

A measured lamination on $S$ is defined as a geodesic lamination equipped with a transverse measure. 
We denote the space of all measured laminations on $S$ by $\ml(S)$. 
By projectivizing the space of measured laminations, 
we obtain the space of projective measured laminations, denoted by 
$\pml(S) = (\ml(S) \setminus \{0\}) / \mathbb{R}_{>0}$. 
For a measured lamination $\mu \in \ml(S)$, we denote its projective class by $[\mu] \in \pml(S)$.

Using the hyperbolic length functions, 
Thurston compactified the Teichm\"uller space $\T(S)$ into a closed ball, 
and its boundary $\partial \T(S)$ is canonically identified with $\pml(S)$. 
We denote the Thurston compactification by 
\[
    \overline{\T(S)}^{\mathrm{Th}} = \T(S) \cup \pml(S).
\]
The action of $\mcg(S)$ on $\T(S)$ extends continuously to homeomorphisms of $\overline{\T(S)}^{\mathrm{Th}}$.

For the once-punctured torus, 
$\pml(S_{1,1})$ is naturally identified 
with the extended real line $\hat{\mathbb{R}} = \mathbb{R} \cup \{\infty\} = \partial_\infty \mathbb{H}$, 
which can be identified with the ideal boundary of the Farey tessellation. 
Under this identification, 
the rational points $\hat{\mathbb{Q}}$ correspond precisely to simple closed curves, 
while the irrational points correspond to uniquely ergodic laminations.

An element $\psi \in \mcg(S)$ is called a \emph{pseudo-Anosov map} 
if there exist a real number $\dil(\psi) > 1$, called the dilatation or stretch factor, 
and a pair of transverse measured laminations $\Lambda_+, \Lambda_- \in \ml(S)$, 
called the stable and unstable laminations, respectively, such that 
\[
    \psi(\Lambda_+) = \dil(\psi) \cdot \Lambda_+ \quad \text{and} \quad \psi(\Lambda_-) = \dil(\psi)^{-1} \cdot \Lambda_-. 
\]
Because $\Lambda_\pm$ are uniquely ergodic, 
their projective classes $[\Lambda_+]$ and $[\Lambda_-]$ 
correspond to irrational numbers in $\partial_\infty \mathbb{H} = \hat{\mathbb{R}}$. 
Under the continuous extension of the action of $\psi$ on $\overline{\T(S)}^{\mathrm{Th}}$, 
the points $[\Lambda_+]$ and $[\Lambda_-]$ are attracting and repelling fixed points, respectively. 
Furthermore, the action of $\psi$ induces a hyperbolic translation on the Farey tree, 
and its translation axis is precisely 
the bi-infinite geodesic connecting the repelling and attracting fixed points, which we denote by $[\Lambda_-, \Lambda_+]$.

It is well known that the action of the mapping class group on $\To$ is equivalent to the standard action of $\mathrm{PSL}(2, \mathbb Z)$ on the upper half-plane. 

\begin{thm}[Equivariant identification of $\overline{\To}^{\mathrm{Th}}$ and $\mathbb H \cup \hat{\mathbb R}$] \label{thm:id_t_to_h}
    There exist 
    a homeomorphism $h \colon \overline{\To}^{\mathrm{Th}} \to \mathbb H \cup \hat {\mathbb R}$ 
    and an isomorphism $H \colon \mcg(S_{1,1}) \to \mathrm{PSL}(2, \mathbb Z)$ such that 
	$h \circ \psi = H(\psi) \circ h$ for all $\psi \in \mcg(S_{1,1})$. 
\end{thm} 

\begin{coro}[North--South dynamics of pseudo-Anosov maps] \label{coro:NSdym}
Let $\psi \in \mcg(S)$ be a pseudo-Anosov mapping class, and let
$[\Lambda_+]$ and $[\Lambda_-]$ be its attracting and repelling fixed points
in $\partial \T(S) = \pml(S)$, respectively.
Then the action of $\psi$ on $\overline{\T(S)}^{\mathrm{Th}}$ has North--South dynamics:
for any neighborhoods $U_+$ of $[\Lambda_+]$ and $U_-$ of $[\Lambda_-]$,
there exists an integer $N \in \mathbb N$ such that
\[
    \psi^n\bigl(\overline{\T(S)}^{\mathrm{Th}} \setminus U_-\bigr) \subset U_+
    \quad \text{and} \quad
    \psi^{-n}\bigl(\overline{\T(S)}^{\mathrm{Th}} \setminus U_+\bigr) \subset U_-
\]
for all $n \geq N$.
\end{coro}

Theorem~\ref{thm:id_t_to_h} and Corollary~\ref{coro:NSdym} will be utilized to 
estimate the approximate locations of the pleating laminations 
for the sequence of quasi-Fuchsian groups obtained by iterating a pseudo-Anosov map $\psi$.

\subsection{3-manifolds: mapping tori and quasi-Fuchsian groups}

By Thurston's hyperbolization theorem, the mapping torus of a pseudo-Anosov map $\psi$, defined as
\[
    M_\psi := S \times [0, 1] / (x, 1) \sim (\psi(x), 0),
\]
admits a complete hyperbolic metric of finite volume.

Next, we denote the deformation space of quasi-Fuchsian groups by $\QF(S)$. 
The Bers simultaneous uniformization theorem provides 
a natural biholomorphic parameterization $\QF(S) \cong \T(S) \times \T(S)$. 
Strictly speaking, 
the first component requires the reversed orientation for the parameterization to be biholomorphic, 
but we omit this distinction for simplicity. 
Specifically, for any ordered pair of points $X, Y \in \T(S)$, 
there exists a unique quasi-Fuchsian group $\Gamma$ 
such that the quotient 3-manifold $Q(X, Y) := \mathbb{H}^3 / \Gamma$ is homeomorphic to $S \times \mathbb{R}$, 
and the conformal structures on the components of the conformal boundary
corresponding to $S\times\{-\infty\}$ and $S\times\{+\infty\}$
are $X$ and $Y$, respectively. 
Within this space, the diagonal subset $\Delta := \{ Q(X, X) \mid X \in \T(S) \}$ corresponds to the Fuchsian groups.

\begin{figure}[h]
  \centering
  \includegraphics[width=0.85\linewidth]{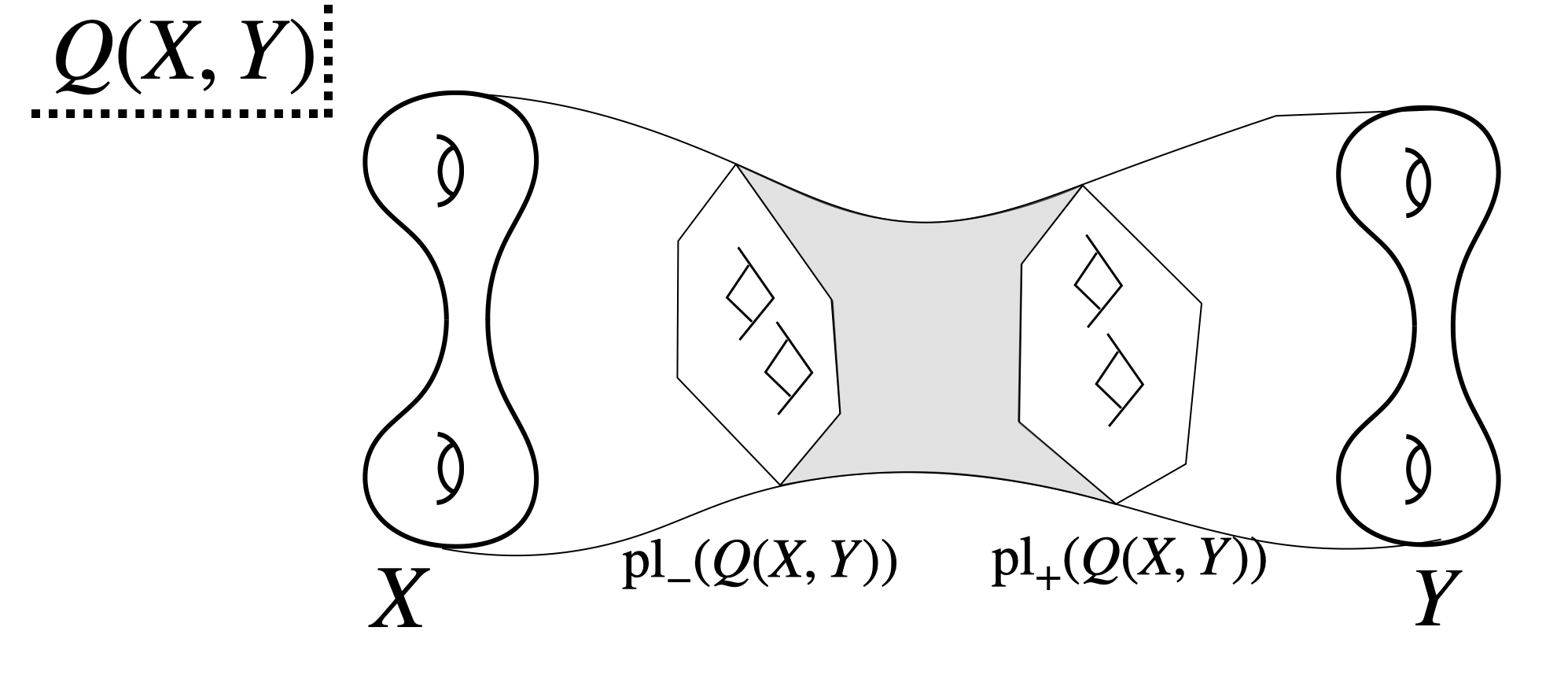}
  \caption{Configuration of the pleating laminations and conformal structures at infinity for a quasi-Fuchsian manifold $Q(X, Y)$.}
  \label{fig:pleatinglami}
\end{figure}

The \emph{convex core} of $Q:=Q(X, Y)$, denoted by $\core(Q)$, is the minimal closed convex submanifold containing all closed geodesics of $Q$. 
Since the quasi-Fuchsian manifold $Q$ is open, 
$Q$ has infinite hyperbolic volume. 
Therefore, it is common to consider the volume of its convex core, 
denoted by $\Vol(\core(Q))$, as a canonical finite-volume substitute. 

The boundary of $\core(Q)$ consists of two connected components, which are pleated surfaces facing the $X$-end and the $Y$-end, respectively. 
A pleated surface is totally geodesic on the complementary regions of a geodesic lamination
and is pleated along the lamination. 
The pleating angle along the leaves defines a transverse measure. 
Consequently, these pleating loci naturally correspond to measured laminations on $S$.
We call them the \emph{pleating laminations}, and denote them by $\pl_-(Q)$ and $\pl_+(Q)$ for the $X$-end and the $Y$-end, respectively. See Figure \ref{fig:pleatinglami}. 
The projective classes $[\pl_-(Q)]$ and $[\pl_+(Q)]$ in $\pml(S)$ 
serve as important geometric invariants that encode the pleating data of the convex core boundary.

In this paper, we study the asymptotic behavior of the convex core volume. 
More precisely, we investigate the sequence of quasi-Fuchsian manifolds obtained by iterating 
the pseudo-Anosov map $\psi$:
\[
    Q_{n,m} := Q(\psi^n(X),\psi^m(X)) \quad (n,m\in\mathbb Z).
\]
For simplicity, we write $Q_n := Q_{-n,n}$ for $n\geq 0$.

By the resolution of the ending lamination conjecture, a hyperbolic
3-manifold with finitely generated fundamental group is determined, up to
isometry, by its end invariants.  In the present setting, these invariants
are either conformal structures at geometrically finite ends or ending
laminations at geometrically infinite ends.  For punctured torus groups, this
principle was established in the form used here by Minsky's classification
theorem \cite{Min99}.  For the manifolds arising from the iterates of the
pseudo-Anosov map $\psi$, the limiting end invariants are given by the stable
and unstable laminations of $\psi$.

We also use the symbols $Q_{n,\infty}$, $Q_{-\infty,m}$, and 
$Q_{-\infty,\infty}$ to denote the corresponding singly or doubly degenerate 
manifolds, where the symbols $\psi^\infty X$ and $\psi^{-\infty}X$ are 
understood as the end invariants $[\Lambda_+]$ and $[\Lambda_-]$, respectively.

When one of the indices is infinite, such as $Q_{0, \infty}$ or $Q_{-\infty, 0}$, the notation represents a singly degenerate 3-manifold determined by the corresponding end invariants, for instance, $X$ and $[\Lambda_+]$. 
Depending on the context, we may also denote these manifolds as $Q(X, [\Lambda_+])$ or $Q([\Lambda_-], X)$ to explicitly emphasize their end invariants.
In the singly degenerate case, 
the corresponding uniformizing Kleinian group is a so-called b-group, 
meaning it possesses a simply connected invariant domain of discontinuity. 
Consequently, the convex hull of the limit set has a well-defined boundary component facing this domain, 
which guarantees that the pleating lamination on that side---such as $\pl_-(Q_{0, \infty})$---is rigorously defined.

Let $\Gamma$ be a finitely generated group. 
Consider a sequence of representations $\{\rho_n \colon \Gamma \to \mathrm{PSL}(2, \mathbb{C})\}$. 
This sequence is said to \emph{converge algebraically} 
to a representation $\rho_\infty$ if $\rho_n(\gamma)$ 
converges to $\rho_\infty(\gamma)$ in $\mathrm{PSL}(2, \mathbb{C})$ 
for every element $\gamma \in \Gamma$.

A sequence of 3-manifolds with basepoint $(M_n, \omega_n)$ \emph{converges geometrically} to 
$(M_\infty, \omega_\infty)$ if, roughly speaking, larger and larger balls around the basepoint $\omega_n$
embed almost isometrically into $M_n$.
More precisely, let $B_R(\omega_\infty)$ denote the metric ball of radius $R$ 
centered at the basepoint $\omega_\infty$ in $M_\infty$.
Then, for any $R > 0$ and $\epsilon > 0$, 
there must exist an integer $N > 0$ such that for all $n > N$, 
there are $(1+\epsilon)$-bi-Lipschitz embeddings $\varphi_n \colon B_R(\omega_\infty) \to M_n$ 
sending $\omega_\infty$ to $\omega_n$. 

Finally, a sequence of representations $\rho_n$ \emph{converges strongly} 
to $\rho_\infty$ if it converges algebraically to $\rho_\infty$ and, 
after choosing suitable basepoints, the corresponding quotient manifolds 
\[
    M_n = \mathbb{H}^3/\rho_n(\Gamma)
\]
converge geometrically to 
\[
    M_\infty = \mathbb{H}^3/\rho_\infty(\Gamma).
\]
For more details, see, e.g., \cite{McM96,BB16}.

\subsection{The Farey tessellation and the Farey tree}

Let us now focus on the once-punctured torus, and let $S = S_{1,1}$. 
After fixing a standard marking of $S$, the set of free homotopy classes of essential simple closed curves on $S$ is identified with the extended rational numbers $\hat{\mathbb{Q}} = \mathbb{Q} \cup \{\infty\}$. 
We represent elements of $\hat{\mathbb{Q}}$ as fractions $p/q$ in lowest terms, with the convention that $\infty = 1/0$. 
The geometric intersection number of two curves corresponding to $p/q$ and $r/s$ is given by $i(p/q, r/s) = |ps - qr|$.

The {\em Farey tessellation} is an ideal triangulation of the upper half-plane model of the hyperbolic plane $\mathbb{H}$. Its ideal vertices are exactly the elements of $\hat{\mathbb{Q}} \subset \partial_\infty \mathbb{H} = \mathbb{R} \cup \{\infty\}$. 
Two vertices $p/q$ and $r/s$ are joined by a geodesic edge in $\mathbb{H}$ 
if and only if their geometric intersection number on $S$ is equal to $1$. 

The graph dual to the Farey tessellation is an infinite trivalent tree, 
known as the {\em Farey tree}. 
The vertices of the Farey tree correspond to the ideal triangles of the Farey tessellation, 
and the edges correspond to the edges of the tessellation.

\subsection{Keen--Series coordinates}

To state the parameterization theorem of $\T(S)$ established by Keen and Series \cite{KeeSer04}, 
we first clarify the definition of pleating invariants. 
For a marked quasi-Fuchsian group $Q \in \QF(S) \setminus \Delta$, 
the pleating laminations $\pl_{\pm}(Q)$ on the boundary 
of its convex core carry natural transverse measures called {\em pleating measures}. 
Let $\mathrm{Len}(\pl_{\pm}(Q))$ be the {\em lamination length} with respect to 
the pleating measures (see \cite{KeeSer04} for the precise definition).
The \emph{pleating invariant} of each boundary component is the pair
\[
    \left([\pl_{\pm}(Q)], \mathrm{Len}(\pl_{\pm}(Q))\right)
    \in \pml(S) \times \mathbb{R}_{>0}.
\]
Since scaling the measure results in a proportional scaling of its length, 
this projective class is canonically identified with an element of $\pml(S) \times \mathbb{R}_{>0}$. 

\begin{thm}[Keen--Series, {\cite[Theorem 1]{KeeSer04}}]\label{thm:ks}
    Let $S$ denote a once-punctured torus. 
    Then, the map $\ks \colon \QF(S) \setminus \Delta \to (\pml(S) \times \mathbb R_{>0})^2$ defined by
    \[
        \ks(Q) :=
        \left(
            \left([\pl_-(Q)], \mathrm{Len}(\pl_-(Q))\right),
            \left([\pl_+(Q)], \mathrm{Len}(\pl_+(Q))\right)
        \right)
    \]
    is a continuous bijection.
\end{thm}

This result states that a pair of pleating invariants 
uniquely determines a marked quasi-Fuchsian punctured-torus group, 
which is analogous to the parameterization by their ending invariants.

\begin{rmk} \label{rmk:KS}
The Keen--Series coordinates naturally extend to the boundary by allowing
the length coordinate to vanish. In this case, the corresponding pleating
lamination is replaced by the ending lamination of the degenerate end.
In particular, the point
\[
    (([\lambda_-],a),([\lambda_+],0)),\qquad a>0,
\]
corresponds to a singly degenerate manifold with ending invariant
$[\lambda_+]$.
\end{rmk}

\subsection{Gu\'eritaud's results} \label{subsec.gueritaud}

We next recall another feature specific to the once-punctured torus, 
concerning ideal triangulations of quasi-Fuchsian manifolds and mapping tori.

We identify $\hat{\mathbb Q}$ with the set of free homotopy classes of simple closed curves 
on the once-punctured torus. Under this identification, 
the pleating laminations of $Q(X, Y)$ can be expressed 
using pairs of real numbers $(\alpha_+, \beta_+)$ and $(\alpha_-, \beta_-)$ as follows.
We represent the projective class of $\pl_\pm(Q(X,Y))$ by the point
$\beta_\pm/\alpha_\pm\in \hat{\mathbb R}$. 
Then we define
\[
    i(\cdot,\pl_\pm(Q(X,Y)))\colon \hat{\mathbb Q}\to\mathbb R_{\ge 0}
\]
by
\[
    i\left(\frac{\xi}{\eta},\pl_\pm(Q(X,Y))\right)
    :=
    \left|
    \det 
    \begin{pmatrix}
        \beta_\pm & \eta \\
        \alpha_\pm & \xi
    \end{pmatrix}
    \right|.
\]
Here, $\xi / \eta \in \hat{\mathbb Q}$ is represented by coprime integers. 

Let $\mathcal E=\{e_i\}_{i\in I}$ be the ordered sequence of edges of the Farey
tessellation crossed by the geodesic
\[
    [\pl_-(Q(X,Y)),\pl_+(Q(X,Y))]
\]
from $[\pl_-(Q(X,Y))]$ to $[\pl_+(Q(X,Y))]$, where $I$ is an interval in $\mathbb Z$.
Here $I$ may be finite, semi-infinite, or bi-infinite, depending on whether the
endpoints of the geodesic are rational or irrational.

For each edge $e_i$, let $q^i_{-} := \xi^i_{-}/\eta^i_{-} \in \hat{\mathbb Q}$ be the vertex of the Farey triangle containing $e_i$ that lies in the connected component separated by $e_i$ containing $\beta_- / \alpha_-$. The fraction $q^i_{+} := \xi^i_{+}/\eta^i_{+} \in \hat{\mathbb Q}$ is defined similarly. Furthermore, we set:
    \begin{equation} \label{def: phi_seq}
        \phi^\pm_{i} := 
        \left| \det 
        \begin{pmatrix}
            \beta_\pm & \eta^i_{\pm} \\
            \alpha_\pm & \xi^i_{\pm}
        \end{pmatrix}
        \right|.
    \end{equation}

According to Gu\'eritaud's construction, the ideal triangulation of the convex core is determined in two distinct stages. 
First, the sequence of edges $\{e_i\}$ crossed by the geodesic in the Farey tessellation completely dictates the \emph{combinatorial structure} of the triangulation, yielding a sequence of abstract topological ideal tetrahedra denoted by $\{ \Delta_i \}$. 
Second, the \emph{geometric realization} of these tetrahedra in hyperbolic space is uniquely determined by a variational principle. 
More precisely, the hyperbolic shape of each tetrahedron $\Delta_i$ is determined by its dihedral angles
\[
    w_{i-1}-w_i,\qquad -w_{i-1},\qquad w_i,
\]
where the sequence $(w_i)$ is the unique maximizer of the volume functional over a specific parameter space constrained by the intersection numbers $\phi_i^\pm$. 
Here and below, we write $\nabla x_i:=x_{i-1}-x_{i}$.
An index $i$ is called a hinge if the two consecutive Farey triangles adjacent 
to $e_i$ lie on opposite sides of the geodesic; otherwise it is called a non-hinge.

This two-step characterization is summarized in the following theorem:

\begin{thm}[Gu\'eritaud, {\cite[Equation (11), Theorem 0.1]{Gue09}}] \label{thm:Gue_geomreal_qf}
    For the ideal triangulation of the convex core of a quasi-Fuchsian manifold $Q:=Q(X, Y)$:
    \begin{itemize}
        \item Its combinatorics are determined by the projective classes $[\pl_\pm(Q)]$ of the pleating laminations and the Farey tessellation.
        \item Its geometric realization is uniquely given by the point that maximizes the volume function over the set
            \[
                W(Q(X, Y)) := \{ 
                (w_i) \mid 
                \text{the following conditions hold} \}.
            \]
        Here, the conditions are given by:
            \[
                \begin{cases}
                    0 < w_i < \min\{ \phi_i^-,\phi_i^+, \pi \} & \text{for all $i$}; \\
                    \nabla \phi^-_i < \nabla w_i < \nabla \phi^+_i & \text{for all $i$}; \\
                    | w_{i + 1} - w_{i-1} | < w_i & \text{if $i$ is a hinge}; \\
                    w_{i + 1} + w_{i-1} < 2w_i & \text{if $i$ is a non-hinge}.
                \end{cases}
            \]
    \end{itemize}
\end{thm}

Having established the relationship between the tetrahedra $( \Delta_i )$ and the intersection numbers $(\phi_i^\pm)$ via the parameter space, the sequence $(\phi_i^\pm)$ provides the following combinatorial and geometric controls over the volume:

\begin{prop}[Gu\'eritaud, {\cite[Lemmas 3.2(i), 3.2(ii),3.2(vii), and 4.1]{Gue09}}] \label{prop:Gue_phi}
    For the sequences $(\phi^\pm_{i})$ and the ideal tetrahedra $( \Delta_i )$ of the convex core, the following hold:
    \begin{enumerate}
        \item The sequence $(\phi ^-_{i})$ is strictly increasing, while $(\phi ^+_{i})$ is strictly decreasing as $i\to+\infty$.
        \item For all \(i\), one has
        \[
            \begin{cases}
                |\phi^\pm_{i+1}-\phi^\pm_{i-1}|=\phi^\pm_i,
                & \text{if \(i\) is a hinge,}\\[0.5ex]
                \phi^\pm_{i-1}+\phi^\pm_{i+1}=2\phi^\pm_i,
                & \text{if \(i\) is a non-hinge.}
            \end{cases}
        \]

        \item For all $i$, one has
        $1 < \phi^-_{i}/\phi^-_{i-1} < 2\ \text{and}\ 1 < \phi^+_{i-1}/\phi^+_i < 2.$
        \item $\displaystyle \lim_{i \to \infty} \phi^+_i = \lim_{i \to - \infty} \phi^-_i = 0$ and $\displaystyle\lim_{i \to \infty} \phi^-_i = \lim_{i \to -\infty} \phi^+_i = \infty$.
        \item There exists a universal constant $C > 0$ (independent of $X, Y \in \T(S)$) such that for any $i \in \mathbb Z$, 
            \[
                \sum_{j \leq i} \Vol(\Delta_j) \leq C \phi^-_i, \quad \sum_{j \geq i} \Vol(\Delta_j) \leq C \phi^+_i.
            \]
    \end{enumerate}
\end{prop}

\begin{rmk}
    Note that the properties of $(\phi_i^\pm)$ stated in 
    Proposition~\ref{prop:Gue_phi}~(1) and~(3) are independent of the geometry 
    of the convex core and depend only on the combinatorics of the Farey tessellation.
    On the other hand, the estimate in Proposition~\ref{prop:Gue_phi}~(5) 
    follows from the concavity of the Lobachevsky volume function.
    Indeed, the proof of \cite[Lemma 4.1]{Gue09} shows that one can choose
    \[
        C = 3(2e^{-1}+1+\log 2)+3\log 2.
    \]
\end{rmk}

The ideal triangulation of the mapping torus $M_\psi$ is obtained in a similar manner:

\begin{thm}[Gu\'eritaud, {\cite[Lemma 6.2, Main Theorem]{Gue06}}]\label{thm:Gue_geomreal_map_torus}
    For the ideal triangulation of the mapping torus:
    \begin{itemize}
        \item Its combinatorics are determined by the stable and unstable laminations $[\Lambda_\pm]$ and the Farey tessellation.
        \item Its geometric realization is uniquely given by the unique critical point of the volume function over the set
            \[
                W(\psi) := \{ 
                (w_i) \in \mathbb R^l \mid 
                \text{the following conditions hold} \}.
            \]
        Here and below, the indices are taken modulo $l$. 
        The conditions are given by:
            \[
                \begin{cases}
                    0 < w_i < \pi & \text{for all $i$}; \\
                    | w_{i + 1} - w_{i-1} | < w_i & \text{if $i$ is a hinge}; \\
                    w_{i + 1} + w_{i-1} < 2w_i & \text{otherwise}.
                \end{cases}
            \]
    \end{itemize}
\end{thm}
    
\begin{rmk}
    Note that the definition of $w_i$ differs by a factor of $2$ between the paper~\cite{Gue09} on quasi-Fuchsian manifolds and the paper~\cite{Gue06} on mapping tori. 
    One should notice that when considering mapping tori, there is no need to consider intersection numbers given as $\phi_i$.
    The combinatorial structure, such as whether an index $i$ corresponds to a hinge or a non-hinge, is completely determined by $\psi$. Based on this information alone, periodicity arises naturally.
\end{rmk}

\begin{prop}
    Let $l$ be the combinatorial translation length of the pseudo-Anosov map $\psi$ acting on the Farey tree. 
    In other words, $l$ is the number of edges of the Farey tree in a fundamental domain of the axis $[\Lambda_-, \Lambda_+]$ of $\psi$.
    Then, the combinatorial structure of the ideal triangulation of $M_\psi$ is periodic with period $l$.
    The elements $q^i_{\pm}(\infty)$ determined by the sequence of Farey triangles connecting $\Lambda_{\pm}$ satisfy the following:
        \[
            \psi^{\pm1}(q^i_{\pm}(\infty)) = q^{i\pm l}_{\pm}(\infty), \quad \psi^{\mp1}(q^i_{\pm}(\infty)) = q^{i\mp l}_{\pm}(\infty).
        \]
\end{prop}

\section{Asymptotics of the convex core volume versus the volume of the mapping torus}
Throughout this section, we work exclusively with the once-punctured torus. 
Hence we omit $S$ from the notation and write
$\T$, $\mcg$, $\pml$, $\ml$, and $\QF$ for the corresponding objects
associated with $S_{1,1}$.

We fix a pseudo-Anosov map $\psi \in \mcg$.
Let $Q_n := Q(\psi^{-n} X, \psi^n X)$ and $\pl_{\pm, n} := \pl_\pm(Q_n)$
with stable and unstable laminations $\Lambda_\pm$.
We also denote the infinite cyclic covering of $M_\psi$ by $\widetilde{M_\psi}$.

\subsection*{Strategy of the proof}
To bound the volume of the convex core from below, 
it suffices to show that the space $W(Q_n)$ in Theorem \ref{thm:Gue_geomreal_qf} contains 
the geometric realization of approximately $2n$ copies of the fundamental domain of $\widetilde{M_\psi}$. 
Since the geometric realization of the convex core of $Q_n$ is given by the weight sequence maximizing the volume function in $W(Q_n)$,
evaluating the volume functional at such a sub-configuration yields a lower bound for the total volume of the convex core.

To obtain an upper bound for the volume of the convex core, 
we reverse the roles of $Q_n$ and $M_\psi$ in the preceding argument. 
With a slight modification, 
it suffices to show that if we divide the ``main portion'' 
(the part occupying the vast majority of the volume) 
of the convex core of $Q_n$ into approximately $2n$ blocks, 
each block can be compared with an admissible period block for $M_\psi$.

\subsection{Combinatorial inflexibility}

When evaluating the lower bound of the convex core volume, 
we must first show that the combinatorial data---namely, 
the sequences of Farey triangles crossed by 
the geodesics $[ \pl_{-, n}, \pl_{+, n} ]$ and $[\Lambda_-, \Lambda_+]$ coincide 
over approximately $2n$ periods.

\begin{lem}\label{lem:pl-lam-control}
    Let $U_\pm$ be neighborhoods of $[\Lambda_\pm]$ in $\pml$. Then there exist open sets $V_\pm \subset \T$ satisfying the following conditions:
    \begin{itemize}
        \item For any $X \in V_-$ and $Y \in V_+$, we have $[\pl_{\pm} (Q(X, Y))] \in U_\pm$.
        \item The closures $\mathrm{cl}(V_\pm)$ in $\overline{\T}^{\mathrm{Th}}$ are neighborhoods of $[\Lambda_\pm]$ in $\pml$, respectively.
    \end{itemize}
\end{lem}

\begin{proof}
    Define
    \[
        \widetilde U_\pm := U_\pm \times \mathbb R_{>0},
        \qquad
        \widetilde U := \widetilde U_- \times \widetilde U_+
        \subset
        (\pml\times \mathbb R_{>0})^2.
    \]

    Since the map $\ks$ given in Theorem \ref{thm:ks} is continuous, the preimage $\ks^{-1} ( \widetilde U ) \subset \QF \setminus \Delta$ is an open set. 
    Let $\mathrm{pr}_\pm \colon \T \times \T \to \T$ be the projections defined by 
        \[
            \mathrm{pr}_\pm (X_-, X_+) := X_\pm.
        \]
    Using the Bers simultaneous uniformization map $\beta \colon \QF \to \T \times \T$, we set
        \[
            V'_\pm := \mathrm{pr}_\pm ( \beta ( \ks^{-1} ( \widetilde U ) ) ) \subset \T.
        \]
%
    We claim that there exists an integer $N \in \mathbb N$ such that the sets $V_\pm := \psi^{\pm N} (V'_{\pm})$
    satisfy $V_- \times V_+ \subset \beta ( \ks^{-1} ( \widetilde U ) )$. 
    
    By Corollary \ref{coro:NSdym}, 
    the sequence of sets $V^{n}_\pm := \psi^{\pm n} (V'_{\pm})$ 
    converges to $[\Lambda_\pm]$ in $\overline{\T}^{\mathrm{Th}}$. 
    Also, by Remark~\ref{rmk:KS}, 
    we have
    \[
        ([\Lambda_-],[\Lambda_+])
        \in
        \overline{\beta(\ks^{-1}(\widetilde U))}
    \]
    in the Thurston compactification.
    Consequently, the product $V^{n}_- \times V^{n}_+$ converges to $( [\Lambda_-], [\Lambda_+] )$ 
    with respect to the product topology. 

    For the once-punctured torus, the Bers boundary is naturally described
    in terms of ending laminations, as in Minsky's model for punctured torus
    groups \cite{Min99}. 
    Under this identification, the convergence results for pseudo-Anosov
    iterations toward the infinite cyclic cover of the mapping torus
    \cite[Theorem 3.11]{McM96}, together with the continuity of the Keen--Series pleating
    coordinates \cite[Theorem 15, Proposition 14]{KeeSer04}, imply that the pleating invariants extend
    continuously to the corresponding degenerate limits.
    In particular, if
    \[
        X_m \to [\Lambda_-],
        \qquad
        Y_m \to [\Lambda_+]
    \]
    in the Thurston compactification, then
    \[
        [\pl_-(Q(X_m,Y_m))] \to [\Lambda_-],
        \qquad
        [\pl_+(Q(X_m,Y_m))] \to [\Lambda_+].
    \]
    Therefore, after replacing $n$ by a sufficiently large integer if necessary,
    we may assume that
    \[
        V_-^n \times V_+^n
        \subset
        \beta\bigl(\ks^{-1}(\widetilde U)\bigr).
    \]
    Indeed, this follows from the product topology on
    $\overline{\T}^{\mathrm{Th}}\times\overline{\T}^{\mathrm{Th}}$:
    the set $\beta(\ks^{-1}(\widetilde U))$ contains the intersection with
    $\T\times\T$ of a product neighborhood of
    $([\Lambda_-],[\Lambda_+])$.
    Thus, setting
    \[
        V_\pm := V_\pm^n
    \]
    for such a sufficiently large $n$, we obtain
    \[
        X\in V_-,\quad Y\in V_+
        \quad\Longrightarrow\quad
        [\pl_\pm(Q(X,Y))]\in U_\pm.
    \]
    Moreover, by the North--South dynamics of $\psi$ on
    $\overline{\T}^{\mathrm{Th}}$, the closures of $V_\pm$ in
    $\overline{\T}^{\mathrm{Th}}$ are neighborhoods of $[\Lambda_\pm]$,
    respectively. 
\end{proof}

We now show 
that the combinatorial data of Gu\'eritaud's triangulation of the convex core of $Q_n$ and 
that of the mapping torus $M_\psi$ coincide over approximately $2n$ periods of $\psi$.
In other words, the combinatorial structures of the ideal triangulation of $Q_n$ 
and the mapping torus $M_\psi$ coincide in the ``deep'' part of the convex core.
Imitating the terminology of geometric inflexibility, 
we refer to such an agreement as \emph{combinatorial inflexibility}.

\begin{figure}[h]
  \centering
  \includegraphics[width=0.55\linewidth]{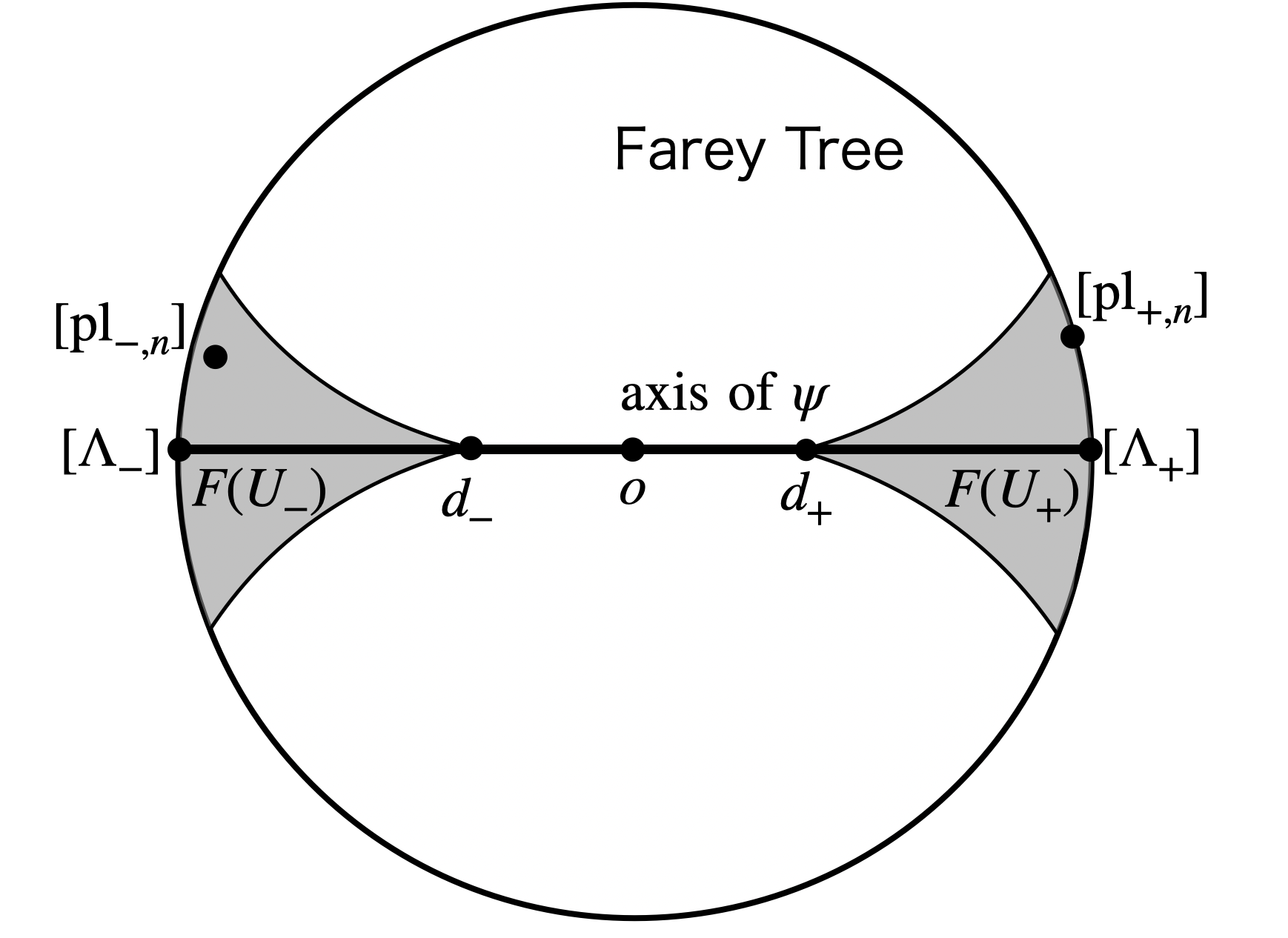}
  \caption{Configuration of the axis of $\psi$, the domains $F(U_\pm)$, 
  and the projective measured laminations $[\pl_{\pm,n}]$ and $[\Lambda_\pm]$.}
  \label{fig:cmbgeominflex}
\end{figure}

\begin{thm}[Combinatorial inflexibility]\label{thm:combinatorial_inflexibility}
    Fix a vertex $o$ on the axis $[\Lambda_-,\Lambda_+]$ of $\psi$ in the Farey tree.
    For sufficiently large $n$, if we increment $n$ by $1$, the length of the
    connected component of
    \[
        [\pl_{-,n},\pl_{+,n}] \cap [\Lambda_-,\Lambda_+]
    \]
    containing $o$ increases by at least two fundamental domains of the action of
    $\psi$ on its axis.
\end{thm}
\begin{proof}

    Let us first recall the map $F$ from $\pml \cong \hat{\mathbb R}$ 
    to the union of the Farey tree and its space of ideal ends.
    For a rational point $p/q \in \hat{\mathbb Q}$, 
    which corresponds to a simple closed curve, 
    $F$ sends it to the corresponding vertex of the Farey tree. 
    For an irrational point $x \in \hat{\mathbb R} \setminus \hat{\mathbb Q}$, 
    which corresponds to a uniquely ergodic lamination, 
    $F$ sends it to the unique ideal end in the Farey tree determined 
    by the infinite sequence of Farey triangles shrinking toward $x$.
    Note that the map $F$ is not continuous but respects continued fractions. 
 
    Under this map, 
    for a sufficiently small open interval around an irrational point, 
    the image under $F$ is contained in a semi-infinite branch of the tree.
    Therefore, we can choose neighborhoods $U_\pm$ of $[\Lambda_\pm]$ in $\pml$ such that  
        \[
            F(U_\pm) \subset \text{a connected component of the Farey tree with a single edge removed}.
        \]
    Let $d_\pm$ denote the endpoints of the edge bounding $F(U_\pm)$ (see Figure \ref{fig:cmbgeominflex}). 
    Notice that $d_\pm$ lie on the translation axis of $\psi$ in the Farey tree. 
    Then, by Lemma \ref{lem:pl-lam-control}, there exist open sets $V_\pm \subset \To$ such that 
        \[
            X \in V_- \text{ and } Y \in V_+ \implies [\pl_{\pm} (Q(X, Y)) ] \in U_\pm. 
        \]
    Because $\mathrm{cl} (V_\pm)$ are neighborhoods of $[\Lambda_\pm]$, there exists an integer $N \in \mathbb N$ such that 
        \[
            \psi^{\pm n} (X) \in V_\pm \quad ( \forall n \geq N ).
        \]
    By the construction of $V_\pm$, we obtain 
        \[
            [\pl_{\pm, n} ] \in U_\pm.
        \]

    Assume $n > N$. 
    The geodesic $[\Lambda_-, \Lambda_+]$ intersects $F(U_\pm)$. 
    Because $[\Lambda_-, \Lambda_+]$ is a geodesic in a tree, the segment $[d_-, d_+]$ is contained in $[\Lambda_-, \Lambda_+]$:
        \[
            [ d_-, d_+]\subset [\Lambda_-, \Lambda_+].
        \]
    Similarly, the condition $[\pl_{\pm, n} ] \in U_\pm$ guarantees that 
    the geodesic connecting $[\pl_{-, n}]$ and $[\pl_{+, n}]$ contains the segment $[d_-, d_+]$:
        \[
            [ d_-, d_+]\subset[\pl_{-, n}, \pl_{+, n}].
        \]
    Consequently, the segment $[d_-, d_+]$ is contained in the intersection of the two geodesics.
        \[
            [ d_-, d_+]\subset[\pl_{-, n}, \pl_{+, n}] \cap [\Lambda_-, \Lambda_+].
        \]

    Next, consider the quasi-Fuchsian manifold $Q_{-n + 1, n + 1}$. The end invariants of the manifold satisfy 
        \[
            \psi^{\pm (n + 1)} (X) \in \psi^{\pm 1}(V_\pm).
        \]
    By the equivariance $\psi(\pl_\pm (Q(X, Y))) = \pl_\pm (Q(\psi(X), \psi(Y)))$, we obtain 
        \[
            [\pl_{\pm}(Q_{-n + 1, n + 1})] \in F(\psi(U_{\pm})) = \psi(F( U_{\pm} ) ).
        \]
    The set $\psi(F(U_\pm))$ is a region bounded by $\psi(d_\pm)$ and contains $\Lambda_\pm$ at infinity. 
    Since geodesics in a tree are unique, the segment $[d_-,d_+]$ is contained in $[\Lambda_-,\Lambda_+]$:  
        \[
            [ \psi(d_-), \psi(d_+)] \subset [\pl_{-}(Q_{-n + 1, n + 1}), \pl_{+}(Q_{-n + 1, n + 1})].
        \]
    Applying $\psi^{-1}$ instead of $\psi$ yields a symmetric inclusion.
        \begin{gather*}
             [\pl_{\pm}(Q_{-n - 1, n - 1})] \in F ( \psi^{-1}(U_{\pm} ) ) = \psi^{-1}( F( U_{\pm} ) ), \\
             [\psi^{-1}(d_-), \psi^{-1}(d_+)] \subset [\pl_{-}(Q_{-n - 1, n - 1}), \pl_{+}(Q_{-n - 1, n - 1})].
        \end{gather*}

    By Corollary \ref{coro:NSdym}, the nesting properties hold:
        \[
            \psi^{-1}(V_-) \subset V_- \subset \psi(V_-), \quad \psi(V_+) \subset V_+ \subset \psi^{-1}(V_+).
        \]
    Combining the nesting properties with the relations $\psi^{-(n+1)}(X) \in \psi^{-1}(V_-)$ and $\psi^{n+1}(X) \in \psi(V_+)$ yields 
        \[
            [ \pl_{\pm}(Q_{n+1}) ] \in \psi^{\pm 1} (F( U_{\pm} )).
        \]
    Consequently, the geodesic $[ \pl_-(Q_{n+1}), \pl_+(Q_{n+1})]$ connects two points contained in $\psi^{-1}(F( U_{-} ))$ and $\psi(F( U_{+} ))$, respectively.
    Therefore, the geodesic must pass through $\psi^{-1}(d_-)$ and $\psi(d_+)$, which implies the following inclusion.
        \[
            [\psi^{-1}(d_-), \psi(d_+)] \subset [ \pl_-(Q_{n+1}), \pl_+(Q_{n+1})]. 
        \]

    Finally, because $d_\pm$ lie on the translation axis of $\psi$, 
    the segment $[ \psi^{-1}(d_-), d_- ]$ is a fundamental domain for the action of $\psi$ on the specified axis (and similarly for $[d_+, \psi (d_+)]$). 
    Furthermore, the inclusion $ [\psi^{-1}(d_-), \psi (d_+)]\subset[\Lambda_-, \Lambda_+]$ yields 
        \begin{align*}
            [ \pl_-(Q_{n+1}), \pl_+(Q_{n+1})] \cap [\Lambda_-, \Lambda_+] & \supset  [\psi^{-1}(d_-), \psi(d_+)] \\
            & =  [\psi^{-1}(d_-), d_-] \cup [ d_-, d_+ ] \cup [ d_+ , \psi(d_+)]. 
        \end{align*}
    The preceding inclusion demonstrates that the intersection length increases by at least two fundamental domains of $\psi$ upon incrementing $n$. 
    
    By choosing $U_-$ and $U_+$ sufficiently small, we may assume that the finite
    segment $[d_-,d_+]$ contains the fixed base vertex $o$.
    Hence the connected component of
    \[
        [\pl_{-,n},\pl_{+,n}]\cap[\Lambda_-,\Lambda_+]
    \]
    containing $o$ contains $[d_-,d_+]$.

\end{proof}

\subsection{Estimates of intersection numbers}
To prove Theorem \ref{thm:main}, we use Gu\'eritaud's results \cite{Gue06,Gue09}.
First, we prepare several estimates for the intersection numbers $\phi_i^\pm$ introduced in Subsection \ref{subsec.gueritaud}.

\begin{lem}\label{lem:intersectionnum-control}
    Let $\gamma$ be an essential simple closed curve. 
    For any $\varepsilon > 0$, there exist neighborhoods $U_\pm$ of $[\Lambda_\pm]$ in $\pml$ such that for any projective class $[\mu_\pm] \in U_\pm$ and any representative measured lamination $\mu_\pm$ of the class $[\mu_\pm]$, we have
        \[
            \left| \log \frac{i( \gamma, \psi \mu_\pm )} {i( \gamma, \mu_\pm )} \mp \log \dil(\psi) \right| < \varepsilon.
        \]
\end{lem}

\begin{proof}
    
    Since $\Lambda_\pm$ are irrational, $i ( \gamma, \Lambda_\pm ) > 0$ for any essential simple closed curve $\gamma$. 
    Therefore, by the continuity and homogeneity of the intersection number, the function
        \[
            f ( [ \mu_\pm ] ) := \log \frac{i( \gamma, \psi \mu_\pm )} {i( \gamma, \mu_\pm )}
        \]
    is a well-defined continuous map on a neighborhood of $[\Lambda_\pm]$ in $\pml$, independent of the choice of representative $\mu_\pm \in [\mu_\pm]$. 
    Furthermore, since $f ( [ \Lambda_\pm ] ) = \pm \log \dil(\psi)$, the conclusion follows directly from the continuity of $f$.
\end{proof}

\begin{prop} \label{prop: growth_in_intersectionnum}
    Let $\gamma$ be an essential simple closed curve. 
    For any $\varepsilon > 0$, there exist neighborhoods $V_\pm \subset \overline{\T}^{\mathrm{Th}}$ of $[\Lambda_\pm]$ such that for any $X \in V_-$ and $Y \in V_+$, we have
        \begin{align*}
             & \left| \log \frac{i( \gamma, \pl_-(Q(X,  Y)) )} {i( \gamma, \pl_-(Q(\psi^{-1}X, Y)))} +
             \log \dil(\psi) \right| < \varepsilon,
             \ \text{and} \\
             & \left| \log \frac{i( \gamma, \pl_+(Q(X, \psi Y) ) )} {i( \gamma, \pl_+(Q(X, Y)))} -
             \log \dil(\psi) \right| < \varepsilon. 
        \end{align*}
\end{prop}

\begin{proof}
    Fix $\varepsilon>0$ and apply Lemma~\ref{lem:intersectionnum-control}
    with $\varepsilon/2$ in place of $\varepsilon$. 
    This gives neighborhoods $U_\pm$ of $[\Lambda_\pm]$ in $\pml$.
    Next, let $V'_\pm$ be the neighborhoods obtained from Lemma \ref{lem:pl-lam-control}, 
    and set $V_\pm := \psi^{\pm 1} (V'_\pm)$. 
    Note that $\psi(V_-) \subset V'_-$. 
    
    As in the proof of Lemma~\ref{lem:pl-lam-control}, we use the boundary
    continuity of pleating laminations for punctured-torus groups. 
    By the results of McMullen~\cite[Theorem 3.11]{McM96} and
    Keen--Series~\cite[Theorem 15, Proposition 14]{KeeSer04}, 
    as $Y \to [\Lambda_+]$ in $\overline{\T}^{\mathrm{Th}}$, we have
        \begin{equation}\label{eq: take nbd}
            \pl_-(Q(\psi^{-1}X, Y)) \to \pl_-(Q(\psi^{-1}X, [\Lambda_+])) 
            \quad \text{and} \quad 
            \pl_-(Q(X, Y)) \to \pl_-(Q(X, [\Lambda_+])).
        \end{equation}
    
    Next, since $\Lambda_+$ is an irrational lamination, it cannot be parallel to the simple closed curve $\gamma$. Because of this transversality between $\gamma$ and $\Lambda_+$, we see that
        \[
             \log \frac{i( \gamma, \pl_-(Q(X, Y)) )} {i( \gamma, \pl_-(Q(\psi^{-1}X, Y)))} \to 
             \log \frac{i( \gamma, \pl_-(Q(X, [\Lambda_+])) )} {i( \gamma, \pl_-( Q(\psi^{-1} X, [ \Lambda_+ ])) )} \quad ( \text{as } Y \to  [ \Lambda_+ ] ).
        \]
    Since both numerator and denominator scale in the same way, 
    the ratio is well-defined on projective classes. 
    
    Furthermore, by equivariance, we have
        \begin{align*}
             \log \frac{i( \gamma, \pl_-(Q(X, [\Lambda_+])) )} 
             {i( \gamma, \pl_-( Q(\psi^{-1} X, [ \Lambda_+ ])) )}  
             &= 
             \log \frac{i( \gamma, \pl_-(Q(X, [\Lambda_+])) )} 
             {i( \gamma, \psi^{-1} \pl_-( Q(X, \psi [ \Lambda_+ ])) )} \\ 
             &= \log \frac{i( \gamma, \pl_-(Q(X, [\Lambda_+])) )} 
             {i( \gamma, \psi^{-1} \pl_-( Q(X, [ \Lambda_+ ])) )}, 
        \end{align*}
    where the last equality follows from the invariance $\psi[\Lambda_+] = [\Lambda_+]$.
    
    By Lemma \ref{lem:intersectionnum-control} applied to $\pl_-(Q(X, [\Lambda_+]))$, we obtain
        \[
            \left| \log \frac{i( \gamma, \pl_-(Q(X, [\Lambda_+])) )} 
             {i( \gamma, \psi^{-1} \pl_-( Q(X, [ \Lambda_+ ])) )} + \log \dil(\psi) \right| < \varepsilon /2.
        \]
    Finally, by the triangle inequality, we conclude
        \begin{align*}
            & \left| \log \frac{i( \gamma, \pl_-(Q(X, Y)) )} {i( \gamma, \pl_-(Q(\psi^{-1}X, Y)))} 
            + \log \dil(\psi)
            \right| \\
            & \leq 
            \left| \log \frac{i( \gamma, \pl_-(Q(X, Y)) )} {i( \gamma, \pl_-(Q(\psi^{-1}X, Y)))} 
            - 
            \log \frac{i( \gamma, \pl_-(Q(X, [\Lambda_+])) )} 
            {i( \gamma, \pl_-( Q(\psi^{-1} X, [ \Lambda_+ ])) )} 
            \right| \\
            & \quad + 
            \left| 
            \log \frac{i( \gamma, \pl_-(Q(X, [\Lambda_+])) )} 
            {i( \gamma, \pl_-( Q(\psi^{-1} X, [ \Lambda_+ ])) )} 
            + \log \dil(\psi)
            \right| \\
            &< \varepsilon,
        \end{align*}
    where $V_+$ is chosen sufficiently small so that, by the convergence established in \eqref{eq: take nbd}, 
    the first term in the sum is bounded by $\varepsilon/2$.
\end{proof}

\begin{lem} \label{lem:independence_of_left_boundary}
    Let $\gamma$ be an essential simple closed curve. Then, 
    \[
        \lim_{n \to \infty} \frac{i(\gamma, \pl_+(Q_{-n-1, n+1}))}{i(\gamma, \pl_+(Q_{-n, n+1}))} = 1.
    \]
\end{lem}

\begin{rmk}
    In other words, for sufficiently large $n$, varying the left boundary condition has asymptotically negligible effect on the pleating lamination of the right boundary.
\end{rmk}

\begin{proof}
    To fix the right boundary condition, we consider the pullback by the mapping class $\psi^{-(n+1)}$.
    Pulling back the two quasi-Fuchsian manifolds $Q_{-n, n+1}$ and $Q_{-n-1, n+1}$ via $\psi^{-(n+1)}$, we obtain
    \begin{align*}
        \psi^{-(n+1)}(Q_{-n, n+1}) &= Q_{-2n-1, 0}, \\
        \psi^{-(n+1)}(Q_{-n-1, n+1}) &= Q_{-2n-2, 0}.
    \end{align*}
    Note that the right boundary component is now fixed to be $X$. 
    
    As $n \to \infty$, both $\psi^{-2n-1}(X)$ and $\psi^{-2n-2}(X)$ converge to the unstable foliation $[\Lambda_-]$ in the Thurston boundary.
    Therefore, the corresponding sequence of Kleinian groups converges strongly to a limit group $\Gamma_\infty$ (\cite[Theorem 3.12]{McM96}).
    
    The limit group $\Gamma_\infty$ is a singly degenerate b-group, where one end is $X$ and the other is degenerate with ending lamination $[\Lambda_-]$.
    Since the conformal boundary component corresponding to $X$ lies in the domain of discontinuity, 
    the $X$-side boundary of the convex core of the corresponding quotient manifold $M_\infty = \mathbb{H}^3 / \Gamma_\infty$ 
    admits a pleating lamination $\pl_+(Q_{-\infty, 0} ) \in \ml$.
    
    By the continuity of the Keen--Series map, when a sequence in $\QF$ converges to a boundary point, the pleating lamination on the non-degenerate side also converges (\cite[Theorem 15, Proposition 14]{KeeSer04}). 
    Thus, we establish the following convergences in $\ml$:
    \begin{equation} \label{eq: convs in sin deg case}
        \lim_{n \to \infty} \pl_+(Q_{-2n-1, 0}) = \pl_+(Q_{-\infty, 0} )  \quad \text{and} \quad 
        \lim_{n \to \infty} \pl_+(Q_{-2n-2, 0}) = \pl_+(Q_{-\infty, 0} ).
    \end{equation}
    
    Here, using the invariance under the action of the mapping class group, the ratio can be rewritten as follows:
    \[
        \frac{i(\gamma, \pl_+(Q_{-n-1, n+1}))}{i(\gamma, \pl_+(Q_{-n, n+1}))} 
        = \frac{i(\psi^{-(n+1)}\gamma, \pl_+(Q_{-2n-2, 0}))}{i(\psi^{-(n+1)}\gamma, \pl_+(Q_{-2n-1, 0}))}.
    \]
    By the north--south dynamics of pseudo-Anosov maps, $[ \psi^{-(n+1)}\gamma]$ converges to the unstable foliation $[\Lambda_-]$ in $\pml$. 
    Therefore, there exists a sequence of positive numbers $c_n$ such that $c_n \psi^{-(n+1)}\gamma$ converges in $\ml$ 
    to a non-zero measured lamination $\bar{\Lambda}_-$ representing $[\Lambda_-]$.
    
    Using the homogeneity of the intersection number, the ratio satisfies
    \begin{equation*}
        \frac{i(\psi^{-(n+1)}\gamma,\pl_+(Q_{-2n-2, 0}))}{i(\psi^{-(n+1)}\gamma,  \pl_+(Q_{-2n-1, 0}))} 
        = \frac{i(c_n \psi^{-(n+1)}\gamma,\pl_+(Q_{-2n-2, 0}))}{i(c_n \psi^{-(n+1)}\gamma, \pl_+(Q_{-2n-1, 0}))}.
    \end{equation*}
    
    Combining the convergence $c_n \psi^{-(n+1)}\gamma \to \bar{\Lambda}_-$ with \eqref{eq: convs in sin deg case}, we obtain
    \[
        \lim_{n \to \infty} \frac{i(c_n \psi^{-(n+1)}\gamma, \pl_+(Q_{-2n-2, 0}))}{i(c_n \psi^{-(n+1)}\gamma, \pl_+(Q_{-2n-1, 0}))} 
        = \frac{i(\bar{\Lambda}_-, \pl_+(Q_{-\infty, 0}))}{i(\bar{\Lambda}_-, \pl_+(Q_{-\infty, 0}))}.
    \]
    
    It remains to show that
    \[
        i(\bar{\Lambda}_-,\pl_+(Q_{-\infty,0}))>0.
    \]
    Suppose, to the contrary, that this intersection number is zero. 
    For each $k\ge 0$, by the equivariance of the pleating lamination under the mapping class group, we have
    \[
        \psi^k\pl_+(Q_{-\infty,0})
        =
        \pl_+(Q_{-\infty,k}).
    \]
    Moreover, since $\psi^k\bar{\Lambda}_-$ is a positive scalar multiple of $\bar{\Lambda}_-$, the invariance and homogeneity of the intersection number imply that
    \[
        i(\bar{\Lambda}_-,\pl_+(Q_{-\infty,k}))=0
    \]
    for every $k\ge 0$. As $k\to\infty$, the sequence $\psi^k X$ converges to $[\Lambda_+]$ in the Thurston compactification. 
    By the continuity of the Keen--Series pleating map on the non-degenerate side, we have
    \[
        [\pl_+(Q_{-\infty,k})]\to [\Lambda_+]
        \quad\text{in }\pml.
    \]
    Choosing positive constants $a_k$ so that
    \[
        a_k\pl_+(Q_{-\infty,k})\to \bar{\Lambda}_+
    \]
    in $\ml$, we obtain by continuity of the intersection number
    \[
        0
        =
        \lim_{k\to\infty}
        i(\bar{\Lambda}_-,a_k\pl_+(Q_{-\infty,k}))
        =
        i(\bar{\Lambda}_-,\bar{\Lambda}_+).
    \]
    This contradicts the fact that 
    the stable and unstable laminations of a pseudo-Anosov mapping class 
    have positive intersection number. 
    Therefore
    \[
        i(\bar{\Lambda}_-,\pl_+(Q_{-\infty,0}))>0.
    \]

\end{proof}

\begin{prop}\label{prop: growth_in_num}
    Let $\gamma$ be an essential simple closed curve. 
    For any $\varepsilon > 0$, there exists an integer $N \in \mathbb N$ such that for all $n \geq N$,
        \[
            \left| \log \frac{i( \gamma, \pl_{\pm, n+1} )} {i( \gamma, \pl_{\pm, n} )}- \log \dil(\psi) \right| < \varepsilon.
        \]
\end{prop}

\begin{proof}
    For sufficiently large $n$, the dynamical properties of $\psi$ on the Thurston boundary guarantee that 
    $\psi^{-n}X \in V_-$ and $\psi^n X \in V_+$, satisfying the requirement of Proposition \ref{prop: growth_in_intersectionnum}.

    Finally, we rewrite the ratio as follows:
        \begin{align*}
            \frac{i( \gamma, \pl_{+, n+1} )} {i( \gamma, \pl_{+, n} )} 
            &= \frac{i( \gamma, \pl_{+}(Q_{-n-1, n+1}) )} {i( \gamma, \pl_{+}(Q_{-n, n}) )} \\
            &= \frac{i( \gamma, \pl_{+}(Q_{-n-1, n+1}) )} {i( \gamma, \pl_{+}(Q_{-n, n+1}) )}
            \cdot \frac {i( \gamma, \pl_{+}(Q_{-n, n+1}) )} {i( \gamma, \pl_{+}(Q_{-n, n}) )}.
        \end{align*}
    The conclusion then follows 
    by applying Lemma \ref{lem:independence_of_left_boundary} to the first factor 
    and Proposition \ref{prop: growth_in_intersectionnum} to the second factor.

    The case of $\pl_{-,n}$ is proved analogously, 
    with the roles of the two ends interchanged.
\end{proof}

\begin{defi}[Scope of combinatorial inflexibility]
\label{def: scope_of_inflexibility}
By Theorem \ref{thm:combinatorial_inflexibility}, the number of fundamental
domains of $\psi$ contained in the intersection of the geodesics
\[
    [\pl_{-,n},\pl_{+,n}]
    \quad\text{and}\quad
    [\Lambda_-,\Lambda_+]
\]
grows by at least two each time $n$ is incremented. Consequently, after
increasing the error constant if necessary, there exists an integer
$B\geq 1$, independent of $n$, such that, for all sufficiently large $n$,
this intersection contains a symmetric segment consisting of at least
\[
    2n-2B=2(n-B)
\]
fundamental domains of $\psi$.

We define
\[
    n':=n-B.
\]
Thus, $n'$ is the number of fundamental domains that can be safely retained
on each side of the midpoint.
We call $n'$ {\em the scope of combinatorial inflexibility}.

Within this scope, we index the edges $\{e_i\}$ crossed by the geodesic
$[\pl_-(Q_n),\pl_+(Q_n)]$ so that the subsegment lying on the axis of $\psi$
corresponds to
\[
    i\in\{-n'l,-n'l+1,\ldots,-1,0,1,\ldots,n'l\}.
\]

For $Q_n=Q_{-n,n}$, let $\phi_i^\pm(n)$ denote the intersection numbers
defined in Subsection \ref{subsec.gueritaud}.
\end{defi}

We first bound the intersection numbers at the endpoint indices
$\pm n'l$ of the scope of combinatorial inflexibility.

\begin{lem}
\label{lem:boundary_intersection_bound}
There exist positive constants $C_{\min}$ and $C_{\max}$, independent of
$n$, such that, for all sufficiently large $n$,
\[
    C_{\min}
    \leq
    \phi_{-n'l}^-(n)
    \leq
    C_{\max}
\]
and
\[
    C_{\min}
    \leq
    \phi_{n'l}^+(n)
    \leq
    C_{\max}.
\]
Thus, the intersection numbers near the two endpoints of the scope of
combinatorial inflexibility are uniformly bounded from above and below by positive constants.
\end{lem}

\begin{proof}
Recall from Definition~\ref{def: scope_of_inflexibility} that
\[
    n'=n-B.
\]

After increasing $B$ if necessary, we may assume that
\[
    i(q^{Bl}_-,\pl_-(Q_{0,\infty}))>0,
    \qquad
    i(q^{-Bl}_+,\pl_+(Q_{-\infty,0}))>0.
\]
This does not affect the asymptotic estimates and is possible since, as
$B$ varies, the vertices $q^{Bl}_-$ and $q^{-Bl}_+$ each run through
distinct rational vertices. We henceforth fix such a $B$; in particular,
$q^{Bl}_-$ and $q^{-Bl}_+$ are independent of $n$.

Within the scope of combinatorial inflexibility, the sequence of Farey
edges agrees with the periodic sequence determined by the axis
$[\Lambda_-,\Lambda_+]$ of $\psi$. Hence the corresponding vertices satisfy
\[
    q^{j+l}_-=\psi(q^j_-),
    \qquad
    q^{j-l}_+=\psi^{-1}(q^j_+)
\]
whenever the indices involved lie in the scope. Since $n'=n-B$, for all
sufficiently large $n$, iterating these identities $n$ times gives
\[
    \psi^n(q^{-n'l}_-)=q^{Bl}_-,
    \qquad
    \psi^{-n}(q^{n'l}_+)=q^{-Bl}_+.
\]

We first consider $\phi_{-n'l}^{-}(n)$. By the definition of
$\phi_{-n'l}^-(n)$, the invariance of the intersection number, and the
identities above, we have
\begin{align*}
    \phi_{-n'l}^-(n)
    &=
    i(q^{-n'l}_-,\pl_-(Q_{-n,n}))\\
    &=
    i\bigl(\psi^n(q^{-n'l}_-),
           \psi^n(\pl_-(Q_{-n,n}))\bigr)\\
    &=
    i(q^{Bl}_-,\pl_-(Q_{0,2n})).
\end{align*}
As $n\to\infty$, the conformal structure $\psi^{2n}X$ converges to
$[\Lambda_+]$ in the Thurston compactification. Hence, as in the proof of
Proposition~\ref{prop: growth_in_intersectionnum},
\[
    \pl_-(Q_{0,2n})
    \longrightarrow
    \pl_-(Q_{0,\infty})
    \qquad
    \text{as } n\to\infty.
\]
It follows that
\[
    \phi_{-n'l}^-(n)
    \longrightarrow
    i(q^{Bl}_-,\pl_-(Q_{0,\infty})).
\]

By the same argument at $\phi_{n'l}^+(n)$, applying $\psi^{-n}$ instead
of $\psi^n$, we obtain
\[
    \phi_{n'l}^+(n)
    =
    i(q^{-Bl}_+,\pl_+(Q_{-2n,0}))
    \longrightarrow
    i(q^{-Bl}_+,\pl_+(Q_{-\infty,0})).
\]

Therefore, we have
\[
    C_{\min}
    \leq
    \phi_{-n'l}^-(n)
    \leq
    C_{\max},
    \qquad
    C_{\min}
    \leq
    \phi_{n'l}^+(n)
    \leq
    C_{\max}
\]
for all sufficiently large $n$.
\end{proof}

\begin{coro} \label{coro:exponential_growth_of_intersection_numbers}
    For any $\varepsilon > 0$, there exist an integer $N \in \mathbb N$ and a positive constant $C > 0$, independent of $n$, such that for all $n \ge N$ and all indices $i$ with $|i| \le n'l$, the following holds:
    \begin{equation}
        \min \{ \phi_i^+(n), \phi_i^-(n) \} \ge C \left( \dil(\psi)^{\frac{1}{l} - \varepsilon} \right)^{n'l - |i|}.
    \end{equation}
\end{coro}

\begin{proof}
    Fix $\varepsilon > 0$. 
    Choose $\delta>0$ sufficiently small so that
    \[
        (\dil(\psi)-\delta)^{1/l}
        \ge
        \dil(\psi)^{1/l-\varepsilon}.
    \]

    For $\phi^{-}_j$'s, using the invariance of the intersection number, we have
    \[
        \frac{\phi_{j+l}^-(n)}{\phi_j^-(n)}
        =
        \frac{i(q_{j+l}^-,\pl_{-,n})}{i(q_j^-,\pl_{-,n})}
        =
        \frac{i(\psi(q_j^-),\pl_{-,n})}{i(q_j^-,\pl_{-,n})}
        =
        \frac{i(q_j^-,\psi^{-1}\pl_{-,n})}{i(q_j^-,\pl_{-,n})}.
    \]
    Since $[\pl_{-,n}]$ converges to $[\Lambda_-]$ in $\pml$ as
    $n\to\infty$, we may apply Lemma~\ref{lem:intersectionnum-control} to each residue class modulo $l$.
    There are only finitely many such residue classes, so the estimate can be
    made uniform over all of them. Hence, after increasing $N$ if necessary,
    for all $n\ge N$ and all admissible indices $j$ with $j,j+l$ in the scope,
    we have
    \[
        \frac{\phi_{j+l}^-(n)}{\phi_j^-(n)}
        \ge
        \dil(\psi)-\delta.
    \]

    Similarly, after increasing $N$ if necessary, the same argument gives
    \[
        \frac{\phi_{j-l}^+(n)}{\phi_j^+(n)}
        \ge
        \dil(\psi)-\delta
    \]
    for all $n\ge N$ and all admissible indices $j$ with $j,j-l$ in the scope.

    We now estimate $\phi_i^-(n)$. By Proposition \ref{prop:Gue_phi}~(1), the
    sequence $(\phi_i^-(n))_i$ is strictly increasing with respect to $i$.
    By Lemma \ref{lem:boundary_intersection_bound}, there exists
    $C_{\min}>0$, independent of $n$, such that
    \[
        \phi_{-n'l}^-(n)\ge C_{\min}
    \]
    for all sufficiently large $n$.

    Let $i$ be an index with $-n'l\le i\le n'l$, and set
    \[
        k:=\left\lfloor \frac{i+n'l}{l}\right\rfloor.
    \]
    Then
    \[
        -n'l+kl\le i<-n'l+(k+1)l.
    \]
    Since $\phi_i^-(n)$ is increasing in $i$, we have
    \[
        \phi_i^-(n)
        \ge
        \phi_{-n'l+kl}^-(n).
    \]
    Iterating the period-$l$ growth estimate from the left endpoint gives
    \[
        \phi_{-n'l+kl}^-(n)
        \ge
        \phi_{-n'l}^-(n)(\dil(\psi)-\delta)^k.
    \]
    Therefore,
    \[
        \phi_i^-(n)
        \ge
        C_{\min}(\dil(\psi)-\delta)^k.
    \]
    Since
    \[
        k\ge \frac{i+n'l}{l}-1,
    \]
    we obtain
    \[
        \phi_i^-(n)
        \ge
        C_{\min}(\dil(\psi)-\delta)^{-1}
        \left((\dil(\psi)-\delta)^{1/l}\right)^{n'l+i}.
    \]
    By the choice of $\delta$, after absorbing the fixed factor
    $(\dil(\psi)-\delta)^{-1}$ into the constant, there exists a constant
    $C_->0$, independent of $n$, such that
    \[
        \phi_i^-(n)
        \ge
        C_-
        \left(\dil(\psi)^{1/l-\varepsilon}\right)^{n'l+i}.
    \]

        Similarly, there exists \(C_+>0\), independent of \(n\), such that
    \[
        \phi_i^+(n)
        \ge
        C_+
        \left(\dil(\psi)^{1/l-\varepsilon}\right)^{n'l-i}.
    \]

    Combining the two estimates, we have
    \[
        \phi_i^-(n)
        \ge
        C_-
        \left(\dil(\psi)^{1/l-\varepsilon}\right)^{n'l+i},
    \]
    and
    \[
        \phi_i^+(n)
        \ge
        C_+
        \left(\dil(\psi)^{1/l-\varepsilon}\right)^{n'l-i}.
    \]
    Since $|i|\le n'l$, we have
    \[
        n'l+i\ge n'l-|i|,
        \qquad
        n'l-i\ge n'l-|i|.
    \]
    Setting
    \[
        C:=\min\{C_-,C_+\},
    \]
    we obtain
    \[
        \min\{\phi_i^+(n),\phi_i^-(n)\}
        \ge
        C
        \left(\dil(\psi)^{1/l-\varepsilon}\right)^{n'l-|i|}.
    \]
\end{proof}

\begin{lem} \label{lem:gradient_exponential_growth}
    Under the same assumptions as Corollary \ref{coro:exponential_growth_of_intersection_numbers}, 
    there exists a uniform constant $C' > 0$ such that, for all sufficiently large $n$ and all indices $i$ with $|i| \le n'l$, we have
    \[
        \min\{|\nabla \phi_i^-(n)|,|\nabla \phi_i^+(n)|\}
        \ge 
        C' \left( \dil(\psi)^{\frac{1}{l} - \varepsilon} \right)^{n'l - |i|}.
    \]
\end{lem}

\begin{proof}
For each index $i$, let $p_i$ denote the common endpoint of the two
consecutive Farey edges $e_{i-1}$ and $e_i$. By the computation in the proof
of \cite[Lemma 3.2]{Gue09},
\[
    \nabla\phi_i^+(n)=i(p_i,\pl_{+,n}),
    \qquad
    -\nabla\phi_i^-(n)=i(p_i,\pl_{-,n}).
\]
In particular,
\[
    |\nabla\phi_i^\pm(n)|=i(p_i,\pl_{\pm,n}).
\]

We prove the estimate for $\phi_i^-(n)$. Within the scope of combinatorial
inflexibility, the vertices $p_i$ satisfy
\[
    p_{i+l}=\psi(p_i).
\]
Recall that $n'=n-B$. At $\phi_{-n'l}^-(n)$, invariance of the
intersection number gives
\[
    |\nabla\phi_{-n'l}^-(n)|
    =
    i(p_{Bl},\pl_-(Q_{0,2n})).
\]
Since
\[
    \pl_-(Q_{0,2n})\longrightarrow\pl_-(Q_{0,\infty}),
\]
we obtain
\[
    |\nabla\phi_{-n'l}^-(n)|
    \longrightarrow
    i(p_{Bl},\pl_-(Q_{0,\infty}))>0.
\]
The positivity follows from the same Farey-tree argument as in
Lemma~\ref{lem:boundary_intersection_bound}. Hence
$|\nabla\phi_{-n'l}^-(n)|$ is uniformly bounded away from zero.

We may now repeat the period-$l$ iteration in the proof of
Corollary~\ref{coro:exponential_growth_of_intersection_numbers}, replacing
$q_i^-$ by $p_i$. After absorbing the finitely many residue classes modulo
$l$ into the constant, we obtain a constant $C'_->0$ such that
\[
    |\nabla\phi_i^-(n)|
    \geq
    C'_-
    \left(\dil(\psi)^{\frac1l-\varepsilon}\right)^{n'l+i}
\]
for all $|i|\leq n'l$ and all sufficiently large $n$.

The estimate for $\phi_i^+(n)$ is obtained in the same way, starting from
the right endpoint $i=n'l$ and applying $\psi^{-n}$. Thus there exists
$C'_+>0$ such that
\[
    |\nabla\phi_i^+(n)|
    \geq
    C'_+
    \left(\dil(\psi)^{\frac1l-\varepsilon}\right)^{n'l-i}.
\]
Since
\[
    n'l+i\geq n'l-|i|,
    \qquad
    n'l-i\geq n'l-|i|,
\]
the conclusion follows by setting
\[
    C':=\min\{C_-',C_+'\}.
\]
\end{proof}

\begin{rmk} \label{rem:sign_of_nabla}
The inequalities in Proposition \ref{prop:Gue_phi},
\[
    1<\frac{\phi_i^-}{\phi_{i-1}^-}<2,
    \qquad
    1<\frac{\phi_{i-1}^+}{\phi_i^+}<2
\]
imply
\[
    \nabla\phi_i^-<0,
    \qquad
    \nabla\phi_i^+>0.
\]
Hence, whenever
\[
    \min\{|\nabla\phi_i^-|,|\nabla\phi_i^+|\}\ge \pi,
\]
we have
\[
    \nabla\phi_i^- \le -\pi,
    \qquad
    \nabla\phi_i^+ \ge \pi.
\]
\end{rmk}

\section{Bounding the Convex Core Volume}

We are now ready to state our main theorem. 

\begin{thm}[Asymptotics of convex core volumes for $S_{1,1}$]\label{thm:Kojima-McShane}
    Let $S = S_{1,1}$ and let $\psi \colon S \to S$ be a pseudo-Anosov map. Then, 
        \[
            | \Vol(\core(Q_n)) - 2n \Vol(M_\psi) | = O(1) \quad (\text{as\ } n \to \infty ).
        \]
\end{thm}

Before proving Theorem \ref{thm:Kojima-McShane}, 
we state a corollary that follows directly from the commensurability between $S_{1,1}$ and $S_{0,4}$.

\begin{coro}[Asymptotics for $S_{0,4}$]
    Let $S=S_{0,4}$ be a four-punctured sphere, and let $\psi \in \mcg(S_{0,4})$ be a pseudo-Anosov mapping class. 
    For any $X \in \T(S_{0,4})$, let $Q_n = Q(\psi^{-n}X, \psi^n X)$ be the sequence of quasi-Fuchsian manifolds. 
    Then, there exists a constant $C > 0$ independent of $n$ such that
        \[
             | \Vol(\core(Q_n)) - 2n \Vol(M_\psi) | = O(1) \quad (\text{as\ } n \to \infty ).
        \]
\end{coro}

\begin{proof}
    It is a well-known fact that the once-punctured torus $S_{1,1}$ 
    and the four-punctured sphere $S_{0,4}$ are commensurable.
    They share a four-punctured torus as a common finite cover. 
    This topological commensurability induces a natural correspondence 
    between their Teichm\"uller spaces and mapping class groups. 
    By passing to this common finite cover, 
    the asymptotic volume estimate for $S_{0,4}$ can be deduced directly from the $S_{1,1}$ case, 
    as the hyperbolic volume is multiplicative under finite coverings. 
    For more detailed discussions on this commensurability, we refer the reader to Komori--Sugawa \cite{KomSug04}.
\end{proof}

We turn to the proof of Theorem~\ref{thm:Kojima-McShane}, estimating
$\Vol(\core(Q_n))$ from below and above.

\subsection{Bounding the convex core volume from below}

\begin{prop} \label{prop:from_below}
	There exists a constant $C_0 > 0$, independent of $n$, 
	such that for all sufficiently large $n$, 
	the following holds:
	\[
		\Vol(\core(Q_n)) \ge 2n \Vol(M_{\psi}) - C_0.
	\]
\end{prop}

\begin{proof}
    Let $\Delta$ be the canonical ideal triangulation of $M_{\psi}$ obtained by
    Gu\'eritaud (see Subsection~\ref{subsec.gueritaud}).
    The weights corresponding to the complete hyperbolic structure of $M_{\psi}$
    form a periodic sequence
    \[
        \mathbf w^{\mathrm{per}}
        =
        (w_i^{\mathrm{per}})_{i\in\mathbb Z}
    \]
    with period $l$, where $l$ is the number of tetrahedra in one fundamental
    domain.

    Set
    \[
        W_{\max}:=\max_{0\le i<l} w_i^{\mathrm{per}}.
    \]
    Since $\mathbf w^{\mathrm{per}}$ is periodic and
    $w_i^{\mathrm{per}}\in(0,\pi)$ for all $i$, we have
    $0<W_{\max}<\pi$.

    Recall from Definition~\ref{def: scope_of_inflexibility} 
    the half-length $n'$ of the scope of combinatorial inflexibility. 
    Within the scope $|i|\le n'l$, the edges along the axis of $\psi$ are indexed by
    \[
        i\in\{-n'l,-n'l+1,\dots,n'l\}.
    \]

    We need to show that the periodic weights $\mathbf w^{\mathrm{per}}$ satisfy
    the constraint conditions for $Q_n$ over a sufficiently large central region.

    Let
    \[
        K:=\dil(\psi)^{1/l-\varepsilon}>1
    \]
    for a sufficiently small $\varepsilon>0$.
    By Corollary~\ref{coro:exponential_growth_of_intersection_numbers}
    and Lemma~\ref{lem:gradient_exponential_growth}, there exist constants
    $C,C'>0$, independent of $n$, such that for all sufficiently large $n$ and
    all $|i|\le n'l$, we have
    \[
        \min\{\phi_i^+(n),\phi_i^-(n)\}
        \ge
        C K^{n'l-|i|}
    \]
    and
    \[
        \min\{|\nabla\phi_i^+(n)|,|\nabla\phi_i^-(n)|\}
        \ge
        C' K^{n'l-|i|}.
    \]

    Choose a sufficiently large integer $k$, independent of $n$, such that
    \[
        CK^k>W_{\max}
        \qquad\text{and}\qquad
        C'K^k\ge \pi.
    \]
    Then, for all sufficiently large $n$ and all indices satisfying
    $|i|\le n'l-k$, we have
    \[
        \min\{\phi_i^+(n),\phi_i^-(n)\}>W_{\max}
    \]
    and
    \[
        \min\{|\nabla\phi_i^+(n)|,|\nabla\phi_i^-(n)|\}\ge \pi.
    \]

    At this point, one has to construct an admissible sequence
    which agrees with the periodic weight sequence on a large central region
    and is cut off near the two ends while preserving all hinge and non-hinge
    inequalities. Since the construction is technical and independent of the main
    geometric argument, we postpone it to Appendix~\ref{appendix.margin}.

    Applying the construction in Appendix~\ref{appendix.margin}, we obtain an
    admissible weight sequence $\mathbf w^{(n)}$ for $Q_n$ such that
    $\mathbf w^{(n)}$ agrees with $\mathbf w^{\mathrm{per}}$ on a central block
    of length $2nl-O(1)$.

    Therefore, using the volume maximization principle associated with the canonical
    triangulation and comparing the central block with the periodic
    triangulation of the mapping torus, we obtain
    \[
        \Vol(\core(Q_n))
        \ge
        2n\Vol(M_\psi)-C_0
    \]
    for some constant $C_0>0$ independent of $n$.
\end{proof}

\subsection{Bounding the convex core volume from above}
\begin{prop} \label{prop:from_above}
    There exists a constant $C_1 > 0$, independent of $n$, such that for all sufficiently large $n$, the following holds:
    \[
        \Vol(\core(Q_n)) \le 2n \Vol(M_{\psi}) + C_1.
    \]
\end{prop}

\begin{proof}
    In Definition \ref{def: scope_of_inflexibility}, 
    we set $n' := n-B$ to be the half-length of the scope of combinatorial inflexibility.
    We decompose the total volume of the convex core into the central stable part and the two boundary tails:
    \begin{align*}
        \Vol(\core(Q_n)) &= \sum_{-n'l \le j < n'l} \Vol(\Delta_j(n)) + 
        \sum_{j < -n'l} \Vol(\Delta_j(n))  + \sum_{j \ge n'l} \Vol(\Delta_j(n)).
    \end{align*}
    Here, $\Delta_j(n)$ denotes the $j$-th ideal tetrahedron in Gu\'eritaud's 
    triangulation associated with the quasi-Fuchsian manifold $Q_n$; see Subsection~\ref{subsec.gueritaud}.
    
    For the boundary tails, Gu\'eritaud's volume bound (Proposition~\ref{prop:Gue_phi}(5)) implies that there exists a universal constant $C > 0$ such that
    \[
        \sum_{j \leq -n'l} \Vol(\Delta_j(n)) \le C \phi^-_{-n'l}(n) \quad \text{and} \quad \sum_{j \ge n'l} \Vol(\Delta_j(n)) \le C \phi^+_{n'l}(n).
    \]
    By Lemma \ref{lem:boundary_intersection_bound}, the intersection numbers $\phi^-_{-n'l}(n)$ and $\phi^+_{n'l}(n)$ are uniformly bounded by a constant $C_{\max}$ independent of $n$. 
    Thus, the total volume of the two boundary tails is bounded by $2C C_{\max}$.

    We next consider the volume of the central part. 
    Let
    \[
        \mathbf w^{\mathrm{core}}=(w_i)
    \]
    be the weight sequence realizing the complete hyperbolic structure on
    $\core(Q_n)$.

    Choose $k$ as in the preceding admissibility argument, and consider only the
    complete period blocks contained in the safe central region
    \[
        |j|\le n'l-k.
    \]

    Moreover, by Theorem \ref{thm:combinatorial_inflexibility}, the hinge and non-hinge pattern
    on the scope of combinatorial inflexibility agrees with the periodic Farey
    path associated with $\psi$. Therefore, on each complete period block in the
    safe central region, the defining angle inequalities are exactly those of the
    periodic parameter space $W(\psi)$. Since the geometric weights satisfy
    $0<w_i<\pi$, the restriction of $\mathbf w^{\mathrm{core}}$ to each such block
    defines an admissible point of $W(\psi)$:
    \[
        (w_j)_{rl\le j<(r+1)l}\in W(\psi).
    \]
    Since the complete hyperbolic structure on $M_\psi$ maximizes the volume
    functional over $W(\psi)$, each such complete period block contributes at most
    $\Vol(M_\psi)$:
    \[
        \sum_{rl\le j<(r+1)l}\Vol(\Delta_j(n))
        \le
        \Vol(M_\psi).
    \]
    There are at most $2n'$ such blocks. The remaining tetrahedra in the central
    region lie within the two end regions
    \[
        n'l-k<|j|<n'l,
    \]
    together with at most one incomplete period block at each end. Hence the
    number of tetrahedra is bounded above by a constant depending only on $k$ and $l$. Since the
    volume of any ideal tetrahedron is at most $v_3$, their total volume is bounded
    above by a constant $C_{\mathrm{cen}}$, depending only on $k$ and $l$, and in
    particular independent of $n$. Hence
    \[
        \sum_{-n'l\le j<n'l}\Vol(\Delta_j(n))
        \le
        2n'\Vol(M_\psi)+C_{\mathrm{cen}}.
    \]

    Combining this with the boundary-tail estimate and using $n'\le n$, we get
    \[
        \Vol(\core(Q_n))
        \le
        2n\Vol(M_\psi)+C_{\mathrm{cen}}+2CC_{\max}.
    \]
    Setting
    \[
        C_1:=C_{\mathrm{cen}}+2CC_{\max},
    \]
    we obtain
    \[
        \Vol(\core(Q_n))
        \le
        2n\Vol(M_\psi)+C_1.
    \]
    Here $C_{\mathrm{cen}}$ depends only on the fixed constants $k$ and $l$, while
    $C$ and $C_{\max}$ are independent of $n$. Hence $C_1$ is independent of $n$.
\end{proof}

Theorem~\ref{thm:Kojima-McShane} follows immediately from Propositions~\ref{prop:from_below} and~\ref{prop:from_above}.

\section{Geometric Application: Convergence of Ideal Triangulations}

We have studied the asymptotic behavior of the convex core volume for
once-punctured torus groups using the combinatorial properties of the Farey
tessellation and the volume maximization principle due to Gu\'eritaud
\cite{Gue06,Gue09}. The same combinatorial description also gives a geometric
application: the ideal triangulations of the convex cores converge, in a local
sense, to the lifted periodic triangulation of the infinite cyclic cover of the
mapping torus.

We first recall Gu\'eritaud's straightening theorem and fix notation. For a
quasi-Fuchsian once-punctured torus group, the pleating laminations on the two
boundary components of the convex core determine a finite path in the Farey
tree, equivalently a finite sequence of left and right turns. This sequence
defines a combinatorial ideal triangulation, which we denote by
$\mathcal D^{\mathrm{Comb}}$. Let $\mathcal D^{\mathrm{Geom}}$ denote the
corresponding geometric ideal triangulation of the interior of the convex core.

\begin{thm}[{\cite[Theorem 0.1]{Gue09}}] \label{thm.Gu}
    If $V$ is the interior of the convex core of a quasi-Fuchsian punctured
    torus group, then the ideal triangulation $\mathcal D^{\mathrm{Geom}}$ of
    $V$ is the totally geodesic straightening of
    $\mathcal D^{\mathrm{Comb}}$.
\end{thm}

\begin{rmk}
    In the universal cover, Theorem \ref{thm.Gu} says that the combinatorial
    triangulation gives the rule for connecting parabolic fixed points on
    $\partial_\infty\mathbb H^3$, and the geometric triangulation is obtained
    by replacing these combinatorial edges with the corresponding hyperbolic
    geodesics. Thus each tetrahedron is determined by the four parabolic fixed
    points which occur as its ideal vertices.
\end{rmk}

We shall use the following notion of convergence. Suppose that the tetrahedra
of the lifted triangulations are indexed by $\mathbb Z$ after choosing a
central tetrahedron.

\begin{defi}
    We say that a sequence of lifted geometric triangulations
    $\mathcal D_n^{\mathrm{Geom}}$ converges to a lifted triangulation
    $\mathcal D_\infty^{\mathrm{Geom}}$ in the finite-window sense if, for
    every finite interval $I\subset\mathbb Z$, the tetrahedra with indices in
    $I$ have the same combinatorics as the corresponding tetrahedra of
    $\mathcal D_\infty^{\mathrm{Geom}}$ for all sufficiently large $n$, and
    their shape parameters converge to the shape parameters of the
    corresponding limiting tetrahedra.
\end{defi}

\begin{rmk}
    This is the analogue of locally uniform convergence: one first fixes a
    finite window of indices and then lets $n\to\infty$.
\end{rmk}

Combining Gu\'eritaud's totally geodesic straightening with the strong
convergence of Kleinian groups gives the following consequence.

\begin{thm}\label{thm:convTri}
    Let
    \[
        Q_n=Q(X_n,Y_n)
    \]
    be the sequence of quasi-Fuchsian once-punctured torus groups considered
    above, and suppose that the corresponding representations $\rho_n$
    converge strongly to a representation $\rho_\infty$. Suppose moreover that
    $\rho_\infty$ corresponds to the infinite cyclic cover of the mapping torus
    $M_\psi$ of a pseudo-Anosov map
    $\psi:S_{1,1}\to S_{1,1}$.

    Then the lifted geometric triangulations of the convex cores of $Q_n$
    converge, in the finite-window sense, to the lift of Gu\'eritaud's
    geometric ideal triangulation of $M_\psi$ to its infinite cyclic cover.
\end{thm}

\begin{proof}
    Let $\mathcal D_n^{\mathrm{Comb}}$ and
    $\mathcal D_n^{\mathrm{Geom}}$ denote respectively the combinatorial and
    geometric Gu\'eritaud triangulations associated with $Q_n$. Let
    $\mathcal D_\infty^{\mathrm{Geom}}$ denote the lift to the infinite cyclic
    cover of Gu\'eritaud's geometric triangulation of $M_\psi$.

    Since $\rho_\infty$ is the infinite cyclic cover of $M_\psi$, its ending
    invariants are the stable and unstable laminations
    $[\Lambda_+]$ and $[\Lambda_-]$ of $\psi$. In the present sequence, the
    conformal structures $X_n$ and $Y_n$ tend to these laminations in the
    Thurston compactification:
    \[
        X_n\to [\Lambda_-],
        \qquad
        Y_n\to [\Lambda_+].
    \]
    Hence, by Lemma \ref{lem:pl-lam-control}, the pleating invariants
    $[\pl_{-,n}]$ and $[\pl_{+,n}]$ converge projectively to
    $[\Lambda_-]$ and $[\Lambda_+]$, respectively.

    It follows that, after translating the path in the Farey tree so that the central
    block corresponds to the origin, the finite subpaths of the Farey path
    associated with $Q_n$ converge combinatorially to the bi-infinite
    geodesic $[\Lambda_-, \Lambda_+]$ in the Farey tree. 
    More precisely, for every finite interval $I\subset\mathbb Z$, 
    the tetrahedra
    $\{\Delta_i(n)\}_{i\in I}$ have, for all sufficiently large $n$, the same
    hinge and non-hinge pattern as the corresponding tetrahedra in the
    infinite cyclic cover of $M_\psi$, by Lemma \ref{lem:pl-lam-control} and Theorem \ref{thm:combinatorial_inflexibility}.

    It remains to identify the geometric shapes. Since $\rho_n\to\rho_\infty$
    strongly, in particular algebraically, we have
    \[
        \rho_n(\gamma)\to \rho_\infty(\gamma)
    \]
    for every $\gamma\in\pi_1(S_{1,1})$. Therefore, for any finite set of
    parabolic elements, the corresponding parabolic fixed points on
    $\partial_\infty\mathbb H^3$ converge to the parabolic fixed points of the
    limiting representation. For a fixed finite interval $I$, only finitely
    many parabolic fixed points occur as vertices of the tetrahedra
    $\{\Delta_i(n)\}_{i\in I}$. Hence all these vertices converge to the
    corresponding vertices for $\rho_\infty$.

    By Theorem \ref{thm.Gu}, each geometric triangulation
    $\mathcal D_n^{\mathrm{Geom}}$ is the totally geodesic straightening of
    the corresponding combinatorial triangulation
    $\mathcal D_n^{\mathrm{Comb}}$. Thus the shape parameter of each
    tetrahedron is given by the cross-ratio of its four ideal vertices. Since
    the corresponding parabolic fixed points converge, these cross-ratios
    converge as well.

    The limiting four-tuples of vertices are precisely the vertices of the
    lifted Gu\'eritaud triangulation of the infinite cyclic cover of
    $M_\psi$. In particular, they determine non-degenerate ideal tetrahedra.
    Therefore, for every finite interval $I\subset\mathbb Z$, the tetrahedra
    $\{\Delta_i(n)\}_{i\in I}$ converge in shape to the corresponding
    tetrahedra of $\mathcal D_\infty^{\mathrm{Geom}}$. This is exactly
    convergence in the finite-window sense.
\end{proof}

\begin{rmk}
    Theorem \ref{thm:convTri} gives a convergence of the ideal
    triangulations. However, it does not by itself imply geometric inflexibility in the
    sense of Brock--Bromberg \cite[Proposition 3.3]{BB16}. Such a quantitative
    statement would require an estimate on the rate at which the relevant
    parabolic fixed points converge on $\partial_\infty\mathbb H^3$, or
    equivalently on the rate of convergence of the cross-ratios of the ideal
    tetrahedra.
\end{rmk}

\appendix

\section{Construction of the admissible sequences}
\label{appendix.margin}
In this appendix, we construct an admissible sequence that we use to prove Proposition \ref{prop:from_below}. 
The goal is to construct a sequence for the quasi-Fuchsian manifold \(Q_n\) that agrees with the periodic sequence of the mapping torus \(M_\psi\) on a large central region.

First, we introduce a family of admissible sequences.
With a suitable choice of the parameter \(\eta\), we can construct a desired admissible sequence.
\begin{lem}[Admissible sequences]\label{lem:matsuda-eta}
Let
\[
        f_\eta(t):=\eta(1-e^{-t})
        \qquad (t\ge 0),
\]
where \(0<\eta<1\).  Define
\[
        u_i^+ := f_\eta(\phi_i^+).
\]
Then the sequence \(u^+=(u_i^+)_{i}\) satisfies the hinge and non-hinge
inequalities
\[
        |u_{i+1}^+-u_{i-1}^+|<u_i^+
        \qquad
        \text{if \(i\) is a hinge,}
\]
and
\[
        u_{i+1}^+ + u_{i-1}^+ < 2u_i^+
        \qquad
        \text{if \(i\) is a non-hinge.}
\]
Moreover, for every \(i\),
\[
        0<u_i^+<\eta\phi_i^+,
        \qquad
        0<\nabla u_i^+<\eta\nabla\phi_i^+,
\]
where \(\nabla u_i^+:=u_{i-1}^+-u_i^+\) and
\(\nabla\phi_i^+:=\phi_{i-1}^+-\phi_i^+\).

In particular, \(u^+\) is admissible.
Similarly, the sequence
\[
        u_i^- := f_\eta(\phi_i^-)
\]
is admissible. 
\end{lem}

\begin{proof}
We prove the assertion for \(u^+\).  The proof for \(u^-\) is obtained in the
same way after reversing the orientation.

The function \(f_\eta\) is positive, strictly increasing, and strictly concave
on \((0,\infty)\).  Moreover, \(f_\eta(0)=0\) and
\[
        0<f_\eta'(t)=\eta e^{-t}<\eta.
\]
Hence
\[
        0<f_\eta(t)<\eta t
        \qquad (t>0).
\]
Applying this to \(t=\phi_i^+\), we obtain
\[
        0<u_i^+<\eta\phi_i^+.
\]

Since \(\phi^+\) is strictly decreasing, we have
\[
        \nabla u_i^+
        =
        u_{i-1}^+ - u_i^+
        =
        f_\eta(\phi_{i-1}^+) - f_\eta(\phi_i^+) >0.
\]
By the mean value theorem and the inequality \(f_\eta'<\eta\), we also have
\[
        \nabla u_i^+
        <
        \eta(\phi_{i-1}^+-\phi_i^+)
        =
        \eta\nabla\phi_i^+.
\]
Thus
\[
        0<\nabla u_i^+<\eta\nabla\phi_i^+.
\]

It remains to verify the hinge and non-hinge inequalities.

First suppose that \(i\) is a non-hinge.  By Proposition \ref{prop:Gue_phi}~(2) for
\(\phi^+_i\), we have
\[
        \phi_{i-1}^+ + \phi_{i+1}^+
        =
        2\phi_i^+.
\]
Since \(\phi^+\) is strictly decreasing, the two values
\(\phi_{i-1}^+\) and \(\phi_{i+1}^+\) are distinct.  Therefore, by the strict
concavity of \(f_\eta\),
\[
        \frac{
            f_\eta(\phi_{i-1}^+)+f_\eta(\phi_{i+1}^+)
        }{2}
        <
        f_\eta\left(
            \frac{\phi_{i-1}^+ + \phi_{i+1}^+}{2}
        \right)
        =
        f_\eta(\phi_i^+).
\]
Equivalently,
\[
        u_{i-1}^+ + u_{i+1}^+ < 2u_i^+.
\]

Next suppose that \(i\) is a hinge.  By Proposition \ref{prop:Gue_phi}~(2) for \(\phi^+_i\),
\[
        \phi_{i-1}^+ = \phi_i^+ + \phi_{i+1}^+.
\]
Since \(f_\eta\) is strictly increasing and \(\phi^+\) is decreasing, we have
\(u_{i-1}^+>u_{i+1}^+\).  Hence
\[
        |u_{i+1}^+-u_{i-1}^+|
        =
        f_\eta(\phi_{i-1}^+) - f_\eta(\phi_{i+1}^+).
\]
For any \(a,b>0\), we have
\[
\begin{aligned}
        f_\eta(a+b)-f_\eta(b)
        &=
        \eta(1-e^{-(a+b)})-\eta(1-e^{-b})  \\
        &=
        \eta e^{-b}(1-e^{-a})             
        <
        \eta(1-e^{-a})
        =
        f_\eta(a).
\end{aligned}
\]
Applying this with
$a=\phi_i^+$ and $b=\phi_{i+1}^+$, 
we obtain
\[
\begin{aligned}
        f_\eta(\phi_{i-1}^+) - f_\eta(\phi_{i+1}^+)
        &=
        f_\eta(\phi_i^+ + \phi_{i+1}^+)
        -
        f_\eta(\phi_{i+1}^+)  \\
        &<
        f_\eta(\phi_i^+).
\end{aligned}
\]
Therefore
\[
        |u_{i+1}^+-u_{i-1}^+|<u_i^+.
\]

This proves all assertions for \(u^+\).  The proof for \(u^-\) is identical
after reversing the order of the indices.
\end{proof}

Next, we prepare lemmas to concatenate two admissible sequences.  
Recall that we have fixed $Q_n$ and then for any index $i\in\Z$ we can determine whether $i$ is a hinge or non-hinge index for $Q_n$.
\begin{lem}\label{lem:concatenate_admissible_sequences}
Let \(u=(u_i)_{i\in\Z}\) and $v=(v_i)_{i\in\Z}$ be two admissible sequences.
Suppose that there exists an index \(i_0\) such that
\[
    u_{i_0} = v_{i_0}.
\]

Define the new sequence \(w=(w_i)_{i\in\Z}\) by
\[
    w_i :=
    \begin{cases}
        u_i & (i\le i_0), \\
        v_i & (i>i_0)
    \end{cases}
\]
Then, if the hinge or non-hinge condition is satisfied at \(i_0\) for $u_{i_0-1}, u_{i_0} = v_{i_0}, v_{i_0+1}$, 
then \(w\) is admissible.
\end{lem}
\begin{proof}
There is almost nothing to prove.
If $i\neq i_0$, the hinge or non-hinge conditions are satisfied because $u$ and $v$ are admissible.
At $i=i_0$, the hinge or non-hinge condition is satisfied by the assumption.
\end{proof}

Recall that $(\phi_i^+)$ is a strictly decreasing sequence.
\begin{lem}\label{lem:non-hinge_to_hinge}
    For any $\epsilon>0$, there exists a constant $C$ depending only on $\epsilon$ such that the following holds.
    Suppose that 
    \begin{itemize}
        \item $i_0-1$ is a non-hinge index,
        \item $i_0$ is a hinge index, and
        \item $\phi^+_i>C$ for $i \leq i_0+2$.
    \end{itemize}

    Then for any $\eta<\pi$ and $\epsilon<a<f_\eta(\phi^+_{i_0})$
    \[
        w_i :=
        \begin{cases}
            f_\eta(\phi^+_i) & (i<i_0), \\ 
            a & (i=i_0)    \\
            f_{\eta}(\phi^+_i) & (i > i_0)
        \end{cases}
    \]
    satisfies the hinge and non-hinge conditions.
\end{lem}
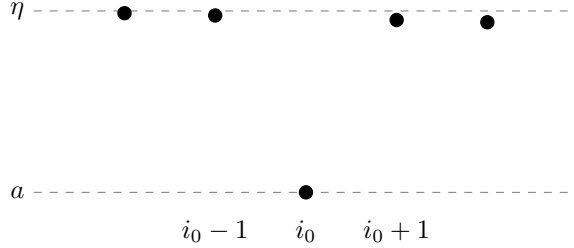
\begin{figure}[ht]
    \begin{center}
\begin{tikzpicture}[x=1.2cm, y=1.5cm, >=stealth]

    \def\iZero{4}          
    \def\Eta{2.8}          
    \def\EtaHalfEps{1.2}   

    \draw[dashed, gray] (0, \Eta) node[left, black] {$\eta$} -- (6, \Eta);
    \draw[dashed, gray] (0, \EtaHalfEps) node[left, black] {$a$} -- (6, \EtaHalfEps);

    \foreach \i in {1,...,2} {
        \filldraw (\i, \Eta-\i*0.02) circle (2.5pt); 
    }
    \filldraw (3, \EtaHalfEps) circle (2.5pt);
    \foreach \i in {4,...,5} {
        \filldraw (\i, \Eta-\i*0.02) circle (2.5pt); 

    }
    \node[below=7pt] at (2, \EtaHalfEps) {$i_0-1$}; 
    \node[below=7pt] at (3, \EtaHalfEps) {$i_0$}; 
    \node[below=7pt] at (4, \EtaHalfEps) {$i_0+1$}; 

\end{tikzpicture}
    \end{center}
    \caption{A sequence defined in Lemma \ref{lem:non-hinge_to_hinge}.}
    \label{fig:non-hinge_to_hinge}
\end{figure}
\begin{proof}
 The sequence $f_\eta(\phi_i^+)$ satisfies hinge and non-hinge conditions when $\phi^{+}_i$'s are all large enough
 (see the proof of Lemma \ref{lem:matsuda-eta}, the condition $\eta<1$ is used only for $f_\eta(\phi_i^{+})<\phi_i^{+}$ which is automatic if $\phi_i^{+}$ is large enough).
Hence, all we need is to check the hinge and non-hinge conditions at $i_0-1$, $i_0$ and $i_0+1$.

Since $i_0-1$ is a non-hinge index and $a<f_\eta(\phi^+_{i_0})$, the non-hinge condition at $i_0-1$ is satisfied:
\[
    w_{i_0-2}+w_{i_0}  = f_\eta(\phi^{+}_{i_0-2}) + a < 2f_\eta(\phi^{+}_{i_0-1}) = 2w_{i_0-1}. 
\]

The hinge condition at $i_0$ is $|w_{i_0+1}-w_{i_0-1}|<a$.
If $C$ is large enough so that $f_\eta(\phi^{+}_{i_0-1})-f_\eta(\phi^{+}_{i_0+1})<\epsilon$, then for any $a$, the condition is satisfied.

We need a case study at $i_0+1$.
\begin{itemize}
    \item If $i_0+1$ is a non-hinge index, then 
    $$
        w_{i_0}+w_{i_0+2}  = a + f_{\eta}(\phi^{+}_{i_0+2}) < 2f_{\eta}(\phi^{+}_{i_0+1})
    $$
    \item If $i_0+1$ is a hinge index, then 
    $$
        |w_{i_0+2}-w_{i_0}| = |f_{\eta}(\phi^{+}_{i_0+2})-a| < f_{\eta}(\phi^{+}_{i_0+1}).
    $$
    should be satisfied for large enough $C$.
\end{itemize}
\end{proof}

We also prepare a lemma to reduce the constant $\eta$.
\begin{lem}\label{lem:reduce_eta}
    For any $\epsilon>0$ with $\epsilon<\eta/2$, there exists a constant $C$ depending only on $\epsilon$ such that the following holds.
    Suppose that $i_0$ is a hinge index and $\phi^{+}_i>C$ for all $i=i_0-1, i_0, i_0+1$.
    Then for any $\eta<\pi$, the sequence defined by
    \[
        w_i :=
        \begin{cases}
            f_\eta(\phi^{+}_i) & (i_0 > i), \\ 
            f_{(\eta/2+\epsilon)}(\phi^{+}_i) & (i \geq i_0)
        \end{cases}
    \]
    satisfies the hinge and non-hinge conditions.
\end{lem}
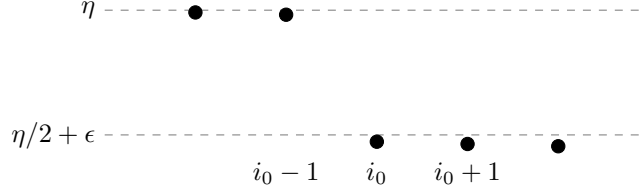
\begin{figure}[ht]
    \begin{center}
\begin{tikzpicture}[x=1.2cm, y=1.5cm, >=stealth]

    \def\iZero{4}          
    \def\Eta{2.8}          
    \def\EtaHalfEps{1.7}   

    \draw[dashed, gray] (0, \Eta) node[left, black] {$\eta$} -- (6, \Eta);
    \draw[dashed, gray] (0, \EtaHalfEps) node[left, black] {$\eta/2 + \epsilon$} -- (6, \EtaHalfEps);

    \foreach \i in {1,...,2} {
        \filldraw (\i, \Eta-\i*0.02) circle (2.5pt); 
    }

    \foreach \i in {3,...,5} {
        \filldraw (\i, \EtaHalfEps-\i*0.02) circle (2.5pt); 
    }
    \node[below=7pt] at (2, \EtaHalfEps) {$i_0-1$}; 
    \node[below=7pt] at (3, \EtaHalfEps) {$i_0$}; 
    \node[below=7pt] at (4, \EtaHalfEps) {$i_0+1$}; 

\end{tikzpicture}
    \end{center}
    \caption{A sequence defined in Lemma \ref{lem:reduce_eta}.}
    \label{fig:half-eta}
\end{figure}
\begin{proof}
We only need to check the hinge and non-hinge conditions at $i_0-1$ and $i_0$.
By choosing $C$ large enough, we may approximate $w_i$'s with error bounds $\epsilon/3$ as follows (see Figure \ref{fig:half-eta}):
\[
w_i\approx
\begin{cases}
    \eta & (i = i_0-2, i_0-1), \\ 
    \eta/2+\epsilon & (i = i_0, i_0 + 1)
\end{cases}
\]
Then the hinge condition at $i_0$ is 
\[
    |w_{i_0+1}-w_{i_0-1}|
    \approx
    |(\eta/2+\epsilon)-\eta|
    =
    \eta/2-\epsilon
    <
    \eta/2 + \epsilon
    \approx
    w_{i_0}.
\]

Similarly, if $i_0-1$ is a hinge index, 
\[
    |w_{i_0}-w_{i_0-2}|
    \approx
    |(\eta/2+\epsilon)-\eta|
    =
    \eta/2-\epsilon
    <
    \eta
    \approx
    w_{i_0-1}.
\]
If $i_0-1$ is a non-hinge index,
\[
    w_{i_0-2}+w_{i_0}
    \approx
    \eta + (\eta/2+\epsilon)
    =
    3\eta/2+\epsilon
    <
    2\eta
    \approx
    2w_{i_0-1}.
\]
This proves the lemma.
\end{proof}

We are now ready to prove the following theorem, which is required in the proof of Proposition \ref{prop:from_below}.
\begin{thm}
    Let $w^\mathrm{per}$ denote the periodic sequence that we obtain from the mapping torus $M_\psi$ in Theorem \ref{thm:Gue_geomreal_map_torus}.
    Recall that $l$ is the period of $w^\mathrm{per}$, and $n'$ is given in Definition \ref{def: scope_of_inflexibility}.
    Then for large enough $n$, there exists $i_0^\pm$ such that $|i_0^\pm|>n'l-k$ for some constant $k$ independent of $n$ so that
    the sequence
    \[
        w_i :=
        \begin{cases}
            w^\mathrm{per}_i & (i_0^- \le i \le i_0^+), \\
            u_i^+ & (i>i_0^+), \\
            u_i^- & (i<i_0^-),
        \end{cases}
    \]
    is admissible for certain admissible sequences $u_i^\pm$.
\end{thm}
\begin{proof}
We only consider $+$-side. The $-$-side is similar.
Let $i_0:=i_0^+$.
First, note that, by Lemma \ref{lem:gradient_exponential_growth}, we have $\nabla \phi_i^{\pm} > \pi$ for large enough $k$.
We consider two cases. 

The first case is when $w^\mathrm{per}$ does not have any non-hinge indices.
In this case, corresponding LR sequence should be LRLRLRLR$\cdots$ and 
hence $w^\mathrm{per}$ is a constant sequence by the volume maximization principle (see \cite{Gue06} for more detail).
Then we may choose $\eta$ so that $w^\mathrm{per}_{i_0}=f_\eta(\phi_{i_0}^+)$ for $i_0=n'l-k+1$.
In this case we should choose $k$ so that by Corollary \ref{coro:exponential_growth_of_intersection_numbers}, 
$\phi_{i_0}^+$ is large enough and the sequence defined by using $\eta$ with $w^\mathrm{per}_{i_0}=f_\eta(\phi_{i_0}^+)$:
\[w_i:=
    \begin{cases}
    w^\mathrm{per}_i & (i \le i_0), \\
    f_\eta(\phi_i^+) & (i>i_0), \\
    \end{cases}
\]
satisfies the hinge condition at $i_0$.
Then by Lemma \ref{lem:concatenate_admissible_sequences}, the sequence is admissible.

If there is a non-hinge index, we may choose $i_0$ to be the first hinge index after a sequence of non-hinge indices.
We may choose $\eta$ so that $w^\mathrm{per}_{i_0-1}=f_\eta(\phi_{i_0-1}^+)$.
Let $\epsilon := \frac{1}{2} \min \{w^{\mathrm{per}}_i\mid 0\leq i \leq l-1 \} (> 0)$. 

\begin{itemize}
    \item If $f_\eta(\phi_{i_0}^+) \leq w^\mathrm{per}_{i_0}$, then since
\[
w^\mathrm{per}_{i_0-2} + f_{\eta}(\phi_{i_0}^+) < 2w^\mathrm{per}_{i_0-1}
\]

the sequence defined as
\[w_i:=
    \begin{cases}
    w^\mathrm{per}_i & (i < i_0), \\
    f_{\eta}(\phi_i^+) & (i\geq i_0),
    \end{cases}
\]
satisfies the hinge and non-hinge conditions.

\item If $f_\eta(\phi_{i_0}^+) > w^\mathrm{per}_{i_0}$, then by Lemma \ref{lem:non-hinge_to_hinge} and Corollary \ref{coro:exponential_growth_of_intersection_numbers}, 
we may choose $k$ so that $\phi_{i_0}^+$ is large enough and the sequence defined as 
\[
w'_i:=
    \begin{cases}
    f_\eta(\phi_i^+) & (i<i_0), \\
    w_{i_0}^\mathrm{per} & (i=i_0), \\
    f_{\eta}(\phi_i^+) & (i>i_0),
    \end{cases}
\]
satisfies the hinge and non-hinge conditions.
The sequence $w_i$ defined above is obtained by concatenating $(w'_i)$ and $(w_{i}^\mathrm{per})$ at $i_0$. 
Therefore, Lemma\ref{lem:concatenate_admissible_sequences} implies that
the sequence $w_i$ also satisfies the hinge and non-hinge conditions.
\end{itemize}
If $\eta$ is bigger than $1$, we may apply Lemma \ref{lem:reduce_eta} at most twice to $w_i$, to construct an admissible sequence.
\end{proof}

\bibliographystyle{alpha} 
\bibliography{references} 
\end{document}